\documentclass[12pt]{article}
\usepackage[english]{babel}
\usepackage{amsmath,amsthm}
\usepackage{amsfonts}

\def\bes{\begin{eqnarray*}}
\def\ees{\end{eqnarray*}}
\def\bee{\begin{eqnarray}}
\def\eee{\end{eqnarray}}
\def\la{\langle}
\def\ra{\rangle}

\def\O{\mathbf O}
\def\idl{\triangleleft}
\def\a{\alpha}

\def\Z{\mathbf Z}
\def\0{\bar 0}
\def\1{\bar 1}

\newcommand{\R}{{\mathbb R}}
\newcommand{\C}{{\mathbb C}}

\newcommand{\VV}{{\mathcal V}}

\def\prf{{\it Proof.\ } }
\def\ctd{\hfill$\Box$}

\def\VV{\mathcal V}

\newtheorem{thm}{Theorem}[section]
\newtheorem{cor}[thm]{Corollary}
\newtheorem{lem}[thm]{Lemma}
\newtheorem{prop}[thm]{Proposition}
\newtheorem{remark}[thm]{Remark}
\newtheorem{conj}[thm]{Conjecture}

\theoremstyle{definition}

\theoremstyle{remark}

\numberwithin{equation}{section}
\begin{document}

\title{Simple unital Jordan superalgebras}

\author{Ivan Shestakov\thanks{Supported by the Ministry of Science and Higher Education of the Republic of Kazakhstan, grant No BR28713025,   by IMC SUSTech, Shenzhen, China, and by Brazilian grants FAPESP 2024/14914-9 and CNPq  305196/2024-3.} \\
{\small Shenzhen International Center for Mathematics, Shenzhen, China}\\
{\small Universidade de S\~ao Paulo}, 
{\small S\~ao Paulo, Brasil}\\
{\small ivan.shestakov@gmail.com}
\and
Efim Zelmanov\thanks{The author gratefully acknowledges support from the NSF of China, grant No 12350710787, and the Guandong Program 2023JC10X085.
}\\
{\small Shenzhen International Center for Mathematics,}\\
 {\small Southern University of Science and Technology, }\\
 {\small Shenzhen, China}\\
 {\small efim.zelmanov@gmail.com}
}


\maketitle
\begin{abstract}
We prove that a simple unital Jordan superalgebra of arbitrary dimension belongs to the list of known simple unital superalgebras or lies in  a certain proper subvariety.
\end{abstract}
\large{
\section{Introduction}
\hspace{\parindent}

Throughout the paper, if otherwise is not stated,  all algebras are considered over a field $F$ of characteristic $\neq 2$. 

A (linear ) {\em Jordan algebra} is a vector space $J$ with a binary bilinear operation $xy$ satisfying the identities
\bes
  xy&=&yx,\\
(x^2y)x&=&x^2(yx).
\ees

A canonical example of a Jordan algebra is the algebra {$A^{(+)}$}  obtained from an  associative algebra {$A$} by introducing the new product {$a\cdot b=\tfrac12(ab+ba)$.}  Algebras of this type and their subalgebras are called {\em special}. 

\smallskip
An important example is a {\em Hermitian algebra} {$H(A,*)$} of self-adjoint elements in an associative algebra {$A$} with involution {$*$}. They include algebras of type {$A^{(+)}$} since {$A^{(+)}\cong H(B,*)$}, where {$B=A\oplus A^{\circ}$} with {$A^{\circ}$} anti-isomorphic to {$A$} and {$(a,b)^*=(b,a)$}.

The Jordan algebra {$H(A,*)$} is simple if and only if {$(A,*)$} is {$*$}-simple.
\smallskip

Another important algebra is an {\em algebra of a bilinear form} {$J(V,f)$} or   {\em Jordan Clifford algebra}. Let {$V$ } be a vector space with a symmetric bilinear form  {$f: V\times V\rightarrow F$}. Consider the vector space direct sum
{$J(V,f)=F\cdot 1\oplus V$} and define a product on it by setting {$1=1_{J(V,f)}$} and for {$u,v\in V$}
\bes
{u\cdot v=f(u,v)1}.
\ees

The algebra {$J(V,f)$} is simple
if and only if  the form {$f$} is non\-de\-ge\-nerate and {$\dim_FV>1$}.

Clearly, it lies in the associative Clifford algebra  {$Cl(V,f)$} of the bilinear form {$f$}, hence it is special.
\smallskip

Not all Jordan algebras are special; the most important example of a non-special (or {\em exceptional})  Jordan algebra is the algebra {$H(\O_3)$} of hermitian {$3\times 3$}-matrices over the algebra {$\O$} of octonions (with respect to the symmetric product {$a\cdot b$}).

The algebra {$H(\O_3)$} is simple. Its exceptionality was  proved in 1934 by A.Albert.  This algebra and its forms are called now { \em Albert algebras}.
\smallskip

It was proved in \cite{Zel6} that\\[1mm]
{\em Every simple Jordan algebra is either special or is isomorphic to a form of the algebra $H(\O_3)$} (an  {\em Albert algebra}).

\medskip

The objective of our paper  are  {\em Jordan superalgebras}.  They were introduced in 1970-s by V.Kac \cite{Kac} and I.Kap\-lan\-sky \cite{Kap1, Kap2}. The class is rich in interesting examples and has deep connections to Lie superalgebras and to Mathematical Physics (see \cite{NS, Ram}).

Consider the definition and main  examples of Jordan superalgebras.
\smallskip

Let $V$ be a vector space of countable dimension and let $G(V)$ be the Grassmann  (or exterior) algebra over $V$. If $\{e_i,\ i\geq 1\}$ is a basis of the space $V$, then the algebra $G(V)$ is presented by generators $e_i,\ i\geq 1$,  and relations $e_ie_j+e_je_i=0,\ i.j\geq 1$. The set of ordered products $1,\, e_i,\, \ldots, e_{i_1}\cdots e_{i_k},\ i_1<i_2<\cdots<i_k,$  is a basis of $G(V)$.

The algebra $G(V)$ is $\Z/2\Z$-graded,  $G(V)=G(V)_{\0}+G(V)_{\1}$. The even part $G(V)_{\0}$ is the span of products of even length: $1,e_{i_1}\cdots e_{i_{2k}},\ i_1<\cdots <i_{2k}$, whereas $G(V)_{\1}$ is the space of all products of odd length $e_{i_1}\cdots e_{i_{2k+1}},\ i_1<\cdots<i_{2k+1}$.

By a {\em superalgebra} we mean an algebra
\bes
A=A_{\0}+A_{\1}
\ees
that is $\Z/2\Z$-graded.   The elements from $A_{\0}$ are usually called {\em even} and the elements from $A_{\1}$ are called {\em odd}.

Given a variety $\VV$ of algebras defined by homogeneous identities (see \cite{Jac, ZSSS}) we say that a superalgebra $A=A_{\0}+A_{\1}$ is a $\VV$-superalgebra if its {\em Grassmann envelope} 
\bes
G(A)=A_{\0}\otimes G(V)_{\0}+A_{\1}\otimes G(V)_{\1}
\ees
belongs to the variety $\VV$.

Knowing the identities defining the variety $\VV$,  one can easily write down the {\em superidentities}  defining the $\VV$-superalgebras.
Thus an associative superalgebra is just a $\Z/2\Z$-graded associative algebra.  A {\em commutative superalgebra}  is a $\Z/2\Z$-graded algebra $A=A_{\0}+A_{\1}$ that satisfies 
the superidentity 
\bes
xy=(-1)^{\bar x\bar y}yx,
\ees
where for an element $a\in A_{\0}\cup A_{\1}$ the symbol  $\bar a$ denotes its parity, $\bar a=0$ or 1.
In particular,  the Grassmann algebra is a commutative superalgebra.
\smallskip


\smallskip

A superalgebra $J=J_{\0}+J_{\1}$ is a {\em Jordan superalgebra} if it satisfies the  superidentities
\bee
& xy=(-1)^{\bar x\bar y}yx,&\label{SJ1}\\
& ((xy)z)t+(-1)^{\bar y\bar z+\bar y\bar t+\bar z\bar t}((xt)z)y+(-1)^{\bar z\bar t}x((yt)z)&\nonumber\\
&= (xy)(zt)+(-1)^{\bar y\bar z}(xz)(yt)+(-1)^{\bar t(\bar y+\bar z)}(xt)(yz).&\label{SJ2}
\eee
Consider the main examples of Jordan superalgebras.
Similarly to the case of algebras,  we have the following three classes of superalgebras:

\smallskip

{\bf Example 1.}
Let $A=A_{\0}+A_{\1}$ be an associative superalgebra. The new operation
\bes
a\circ b=\tfrac12(ab+(-1)^{\bar a\bar b}ba)
\ees
defines a structure of a Jordan superalgebra on $A$. We will denote this Jordan superalgebra as $A^{(+)}$.

{\bf Example 2. } ({\em Hermitian superalgebras}). A linear operator $*:A\rightarrow A$ on a superalgebra $A$ is called a {\em superinvolution} if it satisfies identities $(a^*)^*=a,\ (ab)^*=(-1)^{\bar a\bar b}b^*a^*$ for arbitrary elements $a,b\in A_{\0}\cup A_{\1}$.  If $A$ is associative  then the subspace of symmetric elements $H(A,*)=\{a\in A\,|\, a^*=a\}$ is a Jordan subsuperalgebra of $A^{(+)}$.

\smallskip
{\bf Example 3. }({\em Clifford superalgebra}). Let $V=V_{\0}+V_{\1}$ be a $\Z/2\Z$-graded vector space over $F$ with a bilinear form $(v\,|\,w)$ on $V$ such that $(v\,|\,w)$ is symmetric on $V_{\0}$, skew-symmetric on $V_{\1},$ and $(V_{\0}\,|\,V_{\1})=(0)$.  The direct sum of vector spaces $J(V)=F\cdot 1+V= (F\cdot 1+V_{\0})+V_{\1}$ is a Jordan  superalgebra with respect to multiplication $v\cdot w=(v|w)1.$ We refer to this superalgebra as  {\em superalgebra  of a superform}  and denote $J(V).$
\smallskip

The next examples have no analogues among Jordan algebras.

\smallskip
{\bf Example 4.} The $3$-dimensional {\em Kaplansky superalgebra} $k_3=Fe+(Fx+Fy)$  with the multiplication $e^2=e, \,ex=\tfrac12 x,\, ey=\tfrac12 y,  xy=e$ is simple and not unital.

\smallskip
{\bf Example 5.} The $1$-parametric family of 4-dimensional superalgebras $D_t(F)=(Fe_1+Fe_2)+(Fx+Fy)$, where $e_1,e_2$ are orthogonal even idempotents, $e_ix=\tfrac12 x,\,e_iy=\tfrac12 y,\,xy=e_1+te_2,\, t\in F$. The superalgebra $D_t(F)$ is simple if $t\neq 0$.

\smallskip
{\bf Example 6.} V.\,Kac introduced the 10-dimensional simple superalgebra $K_{10}$ that is related (via the Tits-Kantor-Koecher construction) to the exceptional 40-dimensional Lie superalgebra.

\smallskip
{\bf Example 7.} Let $A$ be an associative commutative superalgebra equipped with a bilinear superskew-symmetric map $[,]: A\times A\rightarrow A, \, [A_{\bar i},A_{\bar j}]\subseteq A_{\overline{i+j}}$. A {\em Kantor double} is a direct sum of vector spaces
\bes
Kan(A,[,])=A+Av, \, |v|=1,
\ees
with the product
\bes
a(vb)=(-1)^{\bar a}vab,\, (vb)a=vba, \, (va)(vb)=(-1)^{\bar a}[a,b].
\ees
We say that $[,]$ is a {\em Jordan bracket} if $Kan(A,[,])$ is a Jordan superalgebra with respect to the grading $Kan(A,[,])_{\0}=A_{\0}+A_{\1}v,\, Kan(A,[,])_{\1}=A_{\1}+A_{\0}v.$
I.\,Kantor \cite{Kantor} showed that every Poisson bracket is a Jordan bracket.  Another example of a Jordan bracket is a {\em bracket of vector type}: let $(A,d)$ be an associative commutative superalgebra with an even derivation $d:A\rightarrow A$, then the bracket $[a,b]=a^db-ab^d$ is Jordan.  D.\,King and K.\,McCrimmon \cite{KingMC} found graded identities that determined the class of Jordan brackets. If the bracket $[,]$ is fixed, then we talk just about superalgebra $Kan(A)$.
\smallskip

N.\,Cantarini and V.\,Kac \cite{CanKac} showed that Jordan brackets are in 1-1 correspondence with {\em Lie contact brackets}.

\smallskip
{\bf Example 8.} Let $A$ be an associative commutative superalgebra equipped with a Jordan bracket $[,]:A\otimes A\rightarrow A$ as in the Example 7.  Suppose further that $A$ has a 
$\Z/2\Z$-grading  $A=A_{(0)}+A_{(1)}$ that is compatible with the grading $A=A_{\0}+A_{\1}$ and, moreover, $[A_{(i)},A_{(j)}]\subseteq A_{(i+j)}$. Then the {\em Kantor double} $Kan(A,[,])$ has the  
$\Z/2\Z$-grading
\bes
Kan(A,[,])&=&Kan(A,[,])_{(0)}+Kan(A,[,])_{(1)},\\
Kan(A,[,])_{(0)}&=&A_{(0)}+A_{(1)}v,\ Kan(A,[,])_{(1)}=A_{(1)}+A_{(0)}v.
\ees
We will refer to $A_{(0)}+A_{(1)}v$ as a {\em twisted Kantor double}.  A twisted Kantor double may be a simple Jordan superalgebra that is not isomorphic to any Kantor double.

For example,  let $A=\R[\sin t,\cos t],$   (see \cite{ZhSh}) where $\R$ is the field of real numbers,  $A_{(0)}=\R[\sin 2t,\cos 2t], \, A_{(1)}=(\sin t)A_{(0)}+(\cos t)A_{(0)},\ A=A_{(0)}+A_{(1)}$ is a $\Z/2\Z$-grading,  $[f(t),g(t)]=f'(t)g(t)-f(t)g'(t))$.
The odd part $A_{(1)}v$ of the twisted Kantor double is a projective module over the even part $A_{(0)}$, it is 2-generated, but not 1-generated. Similar examples over an arbitrary field of zero characteristic were constructed in  \cite{Zhel2}. 
 Moreover,  the examples where the odd part  is an $n$-generated module over the even part but can not be generated by less then $n$ elements were constructed in \cite {ZhZ}  for  abitrary $n\geq 2$ (over fields $\R$ and $\C$).  

\smallskip

{\bf Example 9.} 
Let $(A,d)$ be an associative commutative superalgebra with an even derivation $d:A\rightarrow A$.  Following  \cite{ChengKac},  C.\,Martinez and E.\,Zelmanov introduced a family  of Cheng-Kac Jordan superalgebras  $JCK(A,d)$ that are free $A$-modules of rank 8.  The superalgebras $JCK(F[t,t^{-1}],d/dt)$ are related  via the Tits-Kantor-Koecher construction to the exceptional superconformal algebras $CK_6$ (see \cite{GLS,MarZel}). 

Observe that when $A$ is equipped with a $\Z/2\Z$-grading competible with the derivation $d$,  the superalgebra $JCK(A,d)$ also  contains the twisted subsuperalgebra which is simple and may be not isomorphic to any Cheng-Kac superalgebra \cite{Zhel4}.   We will call it the {\em twisted Cheng-Kac superalgebra}.

\smallskip

V.\,Kac \cite{Kac} (see also I.\,Kantor \cite{Kantor}) proved that every simple finite dimensional Jordan superalgebra over an algebraically closed field $F$ of zero characteristic is isomorphic to one of the superalgebras from the Examples 1 - 8.

\smallskip
In \cite{RZ} it was shown that if $char\,F=p>3$ and the even part $J_{\0}$ is semisimple then every simple finite dimensional Jordan superalgebra is also isomorphic to one of the examples 1 - 8.

If $char\,F=3$ then some new examples appear (see \cite{RZ, Sh3}).

If the even part $J_{\0}$ is not semisimple then the only  new examples are Kantor doubles 
$Kan(O_n\otimes G_m)$ \cite{MarZel} and Cheng-Kac superalgebras $JCK(O_n\otimes G_m,d)$, where $O_n=F[t_1,\ldots,t_n\,|\,t_i^p=0,\, 1\leq i\leq n]$ is the algebra of trancated polynomials, $d$ is an even derivation.

\smallskip

Now let us discuss infinite dimensional simple Jordan algebras and superalgebras.

We will formulate a slightly modified conjecture on classification of simple (finite or infinite dimensional) Jordan superalgebras that is due to N.\,Cantarini and V.\,Kac \cite{CanKac}.

\begin{conj} A simple Jordan superalgebra over a field $F$ of zero characteristic (with a nonzero odd part) is isomorphic to one of the following superalgebras:
\begin{itemize}
\item[I)]
$H(A,*)$, where $(A,*)$ is a $*$-simple associative superalgebra with a superinvolution $*$ and with a noncommutative even part (see \cite{GA, GA-M});
\item[II)]
a Jordan superalgebra of a superform or $k_3$,  or $D_t$,  or $K_{10}$ over some extension of the ground field $F$;
\item[III)]
a Kantor double  $Kan\,(A,[,])$ or a twisted Kantor double.
\item[IV)]
a Jordan Cheng-Kac superalgebra $JCK(A,d)$ or a twisted subsuperalgebra of $JCK(A,d)$.
\end{itemize}
\end{conj}

N.\,Cantarini and V.\,Kac \cite{CanKac} proved the Conjecture for linearly compact simple Jordan superalgebras. V.\,Kac, C.\,Mart\'inez and E.\,Zelmanov \cite{KMZ} proved the Conjecture for superconformal  Jordan algebras, i.e. graded Jordan superalgebras $J=\sum_{i\in\Z}J_i$ that are  simple and have all dimensions $\dim_FJ_i$ uniformarly bounded.

\smallskip

A Jordan superalgebra $J$ is called {\em special} if it is embeddable in $A^{(+)}$ for some associative superalgebra $A$. A superalgebra is {\em $i$-special} if it is a homomorphic image of a special superalgebra. Equivalently, a superalgebra is $i$-special if it satisfies all identities that are satisfied by all special Jordan superalgebras.

It is easy to see that examples of types I -- III, exept for $K_{10}$, are special. Yu.\,Medvedev and E.\,Zelmanov \cite{MedZel1} proved that the superalgebra $K_{10}$ is not even $i$-special. 
K.\,McCrimmon \cite{McC}  (see also \cite{Sh5})  proved that a Kantor double $K(A,[,])$ of a bracket of vector type is special.  He also proved that the Kantor double of the classical Poisson bracket is not special.  On the other hand,  I.\,Shestakov \cite{Sh4, MarShZ} showed that Kantor doubles are always $i$-special.   C.\,Martinez, I.\,Shestakov and E.\,Zelmanov \cite{MarShZ} showed that Jordan Cheng-Kac superalgebras are special.
\medskip

 A weaker form of the above Conjecture can be formulated as follows:
 \smallskip

{\em A simple Jordan superalgebra over a field of zero characteristic with a nonzero odd part is either {${i}$}-special or isomorphic to {$K_{10}$.}}
\smallskip

In \cite{Zel5} the second author constructed a proper variety of Jordan algebras $Var$ satisfying the ideal of identities $X$.  In an arbitrary Jordan superalgebra $J$ the ideal $X(J)$ has bounded multiplicative length (see section 1).  Among finite dimensional simple Jordan algebras only Jordan algebras of symmetric bilinear form 
lie in $Var$. 

In this paper we show that counterexamples to the Conjecture, if any, lie in the variety $Var$. \\

{\bf Main Theorem} {\em Let $J$ be a simple unital Jordan superalgebra over a field of zero characteristic. Then either $J$ is isomorphic to one of the superalgebras  I -- V above or $J\in Var$.}\\

We recall that in a recent paper \cite{ShZ} we classified simple Jordan superalgebras with the even part being a sum of Jordan algebras of nondegenerate symmetric forms.   The proof of lemma 4.4 in \cite{ShZ} contains a mistake; lemmas 4.4 and 4.5 of \cite{ShZ} should be substituted by lemma \ref{lem5.3} below.

\smallskip

The main technical assertion of the paper refers to a wider class of primitive Jordan superalgebras of finite multiplicative length. \\

{\bf Main Proposition} {\em Let $J$ be a unital primitive   Jordan superalgebra with nonzero odd part of finite multiplicative length over an algebraically closed  field of zero characteristic. Then either $J$ is $i$-special or  $J\cong K_{10}$.}\\

\section{Definitions and notation}

We refer to \cite{ShZ,  CanKac} for the principal definitions and notation concerning Jordan superalgebras.  Here we give only definitions and notations that will be used in the paper.
\smallskip

Let $J=J_{\0}+J_{\1}$ be a Jordan superalgebra.  For a homogeneous element $a\in J_{\0}\cup J_{\1}$ we denote (the parity index of $a$) $\bar a=i$ if $a\in J_{\bar i}$.

The defining superidentity \eqref{SJ2} implies the following {\em Jordan  operator identity}
\bee
&R(y)R(z)R(t)+(-1)^{\bar y\bar z+\bar y\bar t+\bar z\bar t}R(t)R(z)R(y)+(-1)^{\bar z\bar t}R((yt)z)&\nonumber \\
&= \!R(y)R(zt)\!+\!(-1)^{\bar y\bar z}R(z)R(yt)\!+\!(-1)^{\bar t(\bar y+\bar z)}R(t)R(yz),& \label{OSJ}
\eee
where  $R(a):x\mapsto xa$ denotes  the operator of right multiplication.

For homogeneous elements $a,b\in J_{\0}\cup J_{\1}$ denote by 
$D(a,b)=R(a)R(b)-(-1)^{\bar a\bar b}R(b)R(a)$ the   graded inner derivation of $J$.
We have the equations
\bee
R(c)D(a,b)-(-1)^{\bar c(\bar a+\bar b)}D(a,b)R(c)&=&R(cD(a,b)),\label{D0}\\
D(ab,c)-D(a,bc)-(-1)^{\bar a \bar b}D(b,ac)&=&0. \label{D1}
\eee
It follows from \eqref{OSJ} that 
\bee\label{D2}
 &R(x)R(y)R(z) \hbox{\em can be represented as  a linear
combination of ope-}&\nonumber \\ 
&\hbox{\em rators of length  2 and of operators of the type
$R(x)D(y, z)$.}&
\eee

A {\em triple Jordan product} of homogeneous elements $a,b,c\in J_{\0}\cup J_{\1}$ is defined as $\{a,b,c\}=(ab)c+a(bc)-(-1)^{\bar b\bar c}(ac)b$.

 For $a\in J_{\0}$ let $U(a)=2R(a)^2-R(a^2)$ be the operator of ``quadratic multiplication'': $xU(a)=\{a,x,a\}$.

An element $0\neq a$ of a Jordan algebra $A$ is called an {\em absolute zero divisor} if $U(a)=0$ or, equivalently, $\{a,A,a\}=(0)$.  The algebra $A$ is called {\em nondegenerate} if it has no absolute zero divisors.

A subspace $K\subseteq J$ is called  an {\em  inner ideal} of $J$ (denoted as $K\triangleleft_{in} J$) if $\{K,J,K\}\subseteq K$.

We call an inner ideal $K$ {\em modular} if for an arbitrary nonzero ideal $I\triangleleft J$ the sum $K+I$  is equal to $J$.  
A unital superalgebra  $J$ is called {\em primitive}  if it contains  a modular inner ideal $K\triangleleft_{in} J,\ K\not=J$.

We say that a (super)algebra $J$ has  {\em a finite multiplicative length  $\leq d$} if for an  arbitrary $m\geq 1$, arbitrary elements $a_1,\ldots,a_m\in J$, the product $R(a_1)\cdots R(a_m)$ of right multiplication operators can be represented as a linear combination of operators $R(b_1)\cdots R(b_k),\, b_i\in J,\, k\leq d$.  In other words, since $J$ is unital, this means that the algebra $R\la J\ra$ of multiplications of $J$ may be written as $R\la J\ra=R(J)^d$.

If $A\subset B$ is a sub(super)algebra of a (super)alebra $B$, we denote by $R_B\la A\ra $ the (nonunital) subalgebra of the algebra $R\la B\ra$ of multiplications of $B$ generated by the set $\{ R_a\,|\, a\in A\}$.


\smallskip




Recall some definitions and results from \cite{Jac, Zel2,ZSSS}.  Though they were formulated for Jordan algebras, the extension to superalgebras is straightforward.

Let  $K\triangleleft_{in}J$ be an inner ideal of a superalgebra $J$.  Define
\bes
Ab(K)=\{k\in K\,|\, k\cdot J\subseteq K\}.
\ees
Then
\bee
&Ab(K)\idl_{in}J,\ Ab(K)\idl K,&\label{id3} \\
&\hbox{the factor-algebra $K/Ab(K)$ is special}.&\label{id5}
\eee

For the benefit of a reader we prove a lemma that was implicitely used in \cite{Zel2}.

\begin{lem}\label{lem3.1}
$Ab(K)^{i+2}\cdot J\subseteq Ab(K)^i,\ i\geq 0$, where $Ab(K)^0=K$.
\end{lem}
\prf
Notice that $KD(Ab(K),J)\subseteq K$, therefore 
$$
Ab(K)D(Ab(K),J)\subseteq Ab(K),
$$
and 
$$
Ab(K)^jD(Ab(K),J)\subseteq Ab(K)^j
$$
 for any $j\geq 0$.

Let us prove the lemma by induction on $i\geq 0$. For $i=0$ the inclusion holds by definition.

Let $i\geq 1$. Then 
$$
Ab(K)^{i+2}\subseteq Ab(K)^{i+1}\cdot Ab(K)+Ab(K)^i\cdot Ab(K)^2.
$$ 
As above,
\bes
Ab(K)^{i+2}\cdot J&\subseteq& (Ab(K)^{i+1}\cdot J)\cdot Ab(K)+Ab(K)^{i+1}D(Ab(K),J)\\
&+& (Ab(K)^i\cdot J)\cdot Ab(K)^2+Ab(K)^iD(Ab(K),Ab(K)\cdot J)\\
&\subseteq &Ab(K)^i
\ees
by the induction assumption.

\ctd
\begin{cor}\label{cor2.0} For an arbitrary $ n\geq 1$  and arbitrary elements  $ a_1,\ldots,a_n\in J$ we have  
\bes
Ab(K)^{2n}R(a_1)\cdots R(a_n)\subseteq K.
\ees
\end{cor}

\section{The structure of the heart $h(J)$ (semisimple part).}
\hspace{\parindent}
From now on  $J$ denotes a unital primitive Jordan superalgebra  over an algebraically closed field $F$ of zero characteristic, $card\,F>\dim_FJ$.
 
\smallskip

In this section we make an assumption that is weaker than finite multiplicative length: {\em there exists $d\geq 1$ such that for any  modular inner ideal $K\triangleleft_{in} J$}
\bee\label{id000}
R\la J\ra=(R_J\la K\ra +F\cdot Id)\, R(J)^d.
\eee
Recall that $R_J\la K\ra$ denotes the (nonunital) subalgebra of  the algebra $R\la J\ra$ of multiplications of $J$ generated by the set $\{ R_a\,|\, a\in K\}$, $Id$ is the identity operator.

Let $S$ denote the $T$-ideal of all $s$-identities in the free Jordan superalgebra.  

We assume  that $S(J)\neq (0)$. 
\begin{lem}\label{lem2.1}
1. For an arbitrary proper modular inner ideal $K\triangleleft_{in} J$ we have $S(K)^{2d}=(0)$.
2. For an arbitrary nonzero ideal $I\triangleleft J$ we have $S(J)^{2d}\subseteq I$.
\end{lem}
\prf
1.  By Corollary \ref{cor2.0},  for arbitrary elements $ a_1,\ldots,a_d\in J$ we have $Ab(K)^{2d}R(a_1)\cdots R(a_d)\subseteq K$.  Then it follows form the condition \eqref{id000} that  $id_J\la Ab(K)^{2d}\ra\subseteq K$, hence  $Ab(K)^{2d}=0$.  Since the factor-algebra $K/Ab(K)$ is special,  we get $S(K)\subseteq Ab(K)$.  Therefore $S(K)^{2d}\subseteq Ab(K)^{2d}=(0)$.

2. Let $(0)\neq I\idl J$. Then $K+I=J$. It follows that $S(J)^{2d}=S(K)^{2d}=(0) \ (mod\, I)$ or, in other words, $S(J)^{2d}\subseteq I$, which completes the proof of the lemma.
\ctd
\begin{lem}\label{lem2.2}
The superalgebra $J$ does not contain nonzero ideals with nil even part.
\end{lem}
\prf
Let $(0)\neq I\idl J, \ I_{\0}$ is nil. Then $K+I=J$. Let $k+b=1,\, k\in K_{\0},\,b\in I_{\0}$. The element $k=1-b$ is invertible wich contradicts the assumption that $K\neq J$.

\ctd

\begin{lem}\label{lem2.2'}
Let $I$ be an ideal of a Jordan (super)algebra $J$.  Define
\bes
 I^{\la 3\ra}=I^3,\, I^{\la 5\ra}=I^{\la 3\ra}\cdot I^2+(I^{\la 3\ra}I)I,\ldots, I^{\la 2k+1\ra}=I^{\la 2k-1\ra}\cdot I^2+(I^{\la 2k-1\ra}I)I.
\ees
Then $I^{\la 2k+1\ra}\idl J,\  I^{\la 2k+1\ra}\subseteq I^{2k+1}\subseteq \left\{\begin{array}{c}
I^{\la k+2\ra},\ k \hbox{ odd},\\ \
I^{\la k+1\ra},  k \hbox{ even}.
\end{array}\right.$.
\end{lem}
\prf
It is easy to see by the induction on $k$ that  $I^{\la 2k+1\ra}\idl J$.  It is clear also that $I^{\la 2k+1\ra}\subseteq I^{2k+1}$.  Let us prove the last inclusions.
We have $I^3=I^{\la 3\ra},\ I^5\subseteq I^{\la 3\ra},$  and for $k>2$
\bes
I^{2k+1}&=&I^{2k-1}I^{2}+(I^{2k-1}I)I+(I^{2k-2}I^2)I\\
&\subseteq& I^{ 2(k-2)+1}I^{2}+(I^{2(k-2)+1}I)I,
\ees
which by induction on $k$ implies the needed inclusions.

\ctd

Recall that {\em the heart $h(J)$} of a superalgebra $J$ is defined as the intersection of all nonzero  ideals of $J$. 
The superalgebra $J$ is subdirectly indecomposible if and only if $h(J)\neq (0)$.

\smallskip
If $S(J)\neq (0)$ then $S(J)^{ k}\neq (0)$ for any $k\geq 1$ by lemma \ref{lem2.2}.  Hence $S(J)^{2d}\subseteq H=h(J)\neq (0)$.
Without loss of generality,  by lemma \ref{lem2.2'} we may assume that  $S(J)^{\la 2d+1\ra}=H$.

From now on we assume that $S(J)\neq (0)$ and that $S(J)^{\la 2d+1\ra}=H$.  For an arbitrary element $b\in J_{\0}$ we define
\bes
{\rm Spec}\,(b)=\{\a\in F\,|\,\a\cdot 1- b\hbox{ is not invertible}\}.
\ees
\begin{lem}\label{lem2.3}
The cardinalities of all spectra $Spec\,(b),\, b\in H_{\0}$ are uniformly bounded.
\end{lem}
\prf
Let $f(x_1,\ldots,x_m)$ be a multilinear element from $S^{\la 2d+1\ra}$, and let $a_1,\ldots,a_m\in J$ be elements such that $f(a_1,\ldots,a_m)\neq 0$. Let $b\in H_{\0}$. We claim that $|Spec\,(b)|< 2m+2$.
Indeed, suppose that $|Spec\,(b)|\geq 2m+2$. Then there exist distinct nonzero scalars $\a_1,\ldots,\a_{2m+1}\in Spec\,(b)$ such that the elements $1-\tfrac{1}{\a_i}b,\,1\leq i\leq 2m+1$, are not invertible.
Let 
$$
p_i(t)=\prod_{j=1}^{2m+1}(1-\tfrac{1}{\a_j}t)/(1-\tfrac{1}{\a_i}t);
$$
 then $ p_1(t),\ldots,p_{2m+1}(t)$ are polynomials with the greatest common divisor 1. Hence there exist polynomials $q_1(t),\ldots,q_{2m+1}(t)$ such that $\sum_{i=1}^{2m+1} p_i(t)q_i(t)=1$. Since $f(a_1,\ldots,a_m)\neq 0$, it follows that there exist $1\leq i_1,j_1,\ldots,i_m,j_m\leq 2m+1$ such that
\bes
a=f(\{p_{i_1}(b)q_{i_1}(b),a_1,p_{j_1}(b)q_{j_1}(b)\},\ldots, \{p_{i_m}(b)q_{i_m}(b),a_m,p_{j_m}(b)q_{j_m}(b)\})\neq 0.
\ees
Let $\mu\in\{1,2,\ldots,2m+1\}\setminus \{i_1,j_1,\ldots,i_m,j_m\}$. Then all the elements $\{p_{i_s}(b)q_{i_s}(b),a_s,p_{j_s}(b)q_{j_s}(b)\}$ lie in $\{1-\tfrac{1}{\a_{\mu}}b,J,1-\tfrac{1}{\a_{\mu}}b\}=K_{\mu}$, and hence $a\in S^{\la 2d+1\ra}(K_{\mu}).$ The inner ideal $K_{\mu}$ is modular. Indeed, if $(0)\neq I\triangleleft J$ then $b\in H_{\0}\subseteq I$, therefore  $K_{\mu}+I=J$.  By part 1 of lemma \ref{lem2.1} 
 $S(K_{\mu})^{\la 2d+1\ra}=S^{\la 2d+1\ra}(K_{\mu})=(0)$, which  contradicts the  assumption $0\neq a\in S^{\la 2d+1\ra}(K_{\mu})$.

\ctd
\begin{cor}
The algebra $H_{\0}$ is algebric over $F$, that is, for every $a\in H_{\0}$ there  exists a polynomial $f(x)\in F[x]$ such that $f(a)=0$.  Besides, if $n=\max\{|Spec\,(b)|, b\in H_{\0}\}$ then the algebra $H_{\0}$ does not contain $n+1$ pairwise orthogonal idempotents.
\end{cor}
\prf
The Amitsur trick (see  \cite[chapter 1]{Jac2}) claims that if 
$Card (F\setminus {\rm Spec}\,(a))> \dim_F H_{\0}$ then the element $a$ is algebraic. 
In view of the condition $Card\, F>\dim_FJ$,  the first statement of the Corollary follows from Lemma \ref{lem2.3}.

The second statement follows from Lemma \ref{lem2.3} and the fact that for a set of pairwise  orthogonal idempotents $e_1,\ldots, e_k$ and a set of different elements $\a_1,\ldots \a_k\in F$, for  the element $a=\a_1 e_1+\cdots+\a_ke_k$ we have ${\rm Spec}\,(a)\supset \{\a_1,\ldots,\a_k\}$.

\ctd

\begin{prop}\label{prop2.4}
Let $A$ be a Jordan algebra over an algebraically closed field $F, \ Nil(A)=(0),\ A$ is algebraic and does not contain $n+1$ pairwise orthogonal idempotents. Then $A\cong A'\oplus A_1\oplus\cdots\oplus A_s$, where $\dim_F A'<\infty,\ A_i$'s are Jordan algebras of symmetric bilinear forms.
\end{prop}

We will need the following lemma.
\begin{lem}\label{lem0}
Let $A$ be a unital Jordan algebra and an element $a\in A$ satisfies the condition $\{a,A,a\}\subseteq F\cdot a$. Then the ideal $id_A\la a\ra$ generated by $a$ in $A$ is  locally finite dimensional.
\end{lem}  
\prf
Consider the ser
$$
X=\{a\in A\,|\,\{a,A,a\}\subseteq Fa\}.
$$
For an arbitrary element $u\in A$ we have
$$
\{u,X,u\}\subseteq X.
$$
Since the algebra $A$ contains 1,  it follows that the linear $F$-span of $X$ is an ideal of $A$.  Now it remains to show that the subalgebra generated by arbitrary elements $a_1,\ldots,a_n\in X$ is finite dimensional.
 
Let us assume that $A$ is generated by the set $M=\{1,a_1,\ldots,a_n\}$.  Consider the ascending filtration on $A$: 
\bes
A_1\subseteq A_2\subseteq\cdots\subseteq A_k\subseteq \cdots \subseteq A=\cup_kA_k,\ \ (*)
\ees
where $A_k=M^k$. Let $gr A$ be the graded algebra associated with filtration (*), that is, $gr A=\oplus_k A_k/A_{k-1}$, with the multiplication defined by $(x+A_i)(y+A_k)=xy+A_{i+k}$.
Then $gr A$ is a Jordan algebra, and the elements $a_i\in A_1,\ i=1,\ldots,k,$ satisfies the condition $\{a_i,gr A,a_i\}=0$, that is, $a$ are {\em absolute zero divisors} in $gr A$. By \cite{Zel1}, the algebra  $gr A$ is nilpotent and finite dimensional.   Since $\dim A=\dim gr A$, this proves the lemma.

\ctd

{\em The proof of the proposition}.
Let $e_1,\ldots,e_r$ be a maximal system of pairwise orthogonal idempotents in $A$, $r\leq n$. Then each $\{e_i,A,e_i\}$ does not contain a proper idempotent. Since it is algebraic, the field $F$ is algebraically closed and $Nil(\{e_i,A,e_i\})=Nil(A)\cap \{e_i,A,e_i\}=(0)$, it follows that $\{e_i,A,e_i\}=F\cdot e_i$. Denote $e=e_1+\cdots +e_r$. Then the Pierce component $AU_{1-e}$ is algebraic and does not contain idempotents, and $Nil(AU_{1-e})=(0)$. This implies that $AU_{1-e}=(0)$. For an arbitrary element $b=\{1-e,x,e\}\in\{1-e,A,e\}$ we have in the associative envelope of the special Jordan algebra 
$alg\la e,x\ra$  
$$
b^2=ex(1-e)xe=(1-e)U_xU_e
$$
 and hence 
 $$
 AU_{b^2}=AU_eU_{x}U_{1-e}U_xU_e=(0).
 $$
  Since the algebra $A$ is nondegenerate, we get $\{1-e,A,e\}^2=(0)$, hence $\{1-e,A,e\}$ is a nilpotent ideal in $A$ and $\{1-e,A,e\}=(0)$, $e$ is the identity element of $A$.

From $\{e_i,A,e_i\}=F\cdot e_i$ it follows by lemma \ref{lem0}  that $e_i$ generates a locally finite dimensional ideal in $A$, hence the algebra $A$ is locally finite dimensional.

Denote by $I_k$ the ideal of $A$ generated by the idempotents $e_k,\ k=1,\ldots,r$.  We claim that if $I_i\cdot I_j\neq (0)$ then the idempotents $e_i,e_j$ are strongly connected. Indeed, if $I_i\cdot I_j\neq (0)$ then there exist a nonnilpotent element $c\in I_i\cdot I_j$. This follows from the fact that $I_i\cdot I_j=(I_i\cdot I_j)\cdot I_j+I_i\cdot I_j^2$ (since $I_k=I_k^2$),  which is a nonzero ideal of $A$ and hence contains a nonnilpotent element (recall that $Nil\, A=(0)$. Let
\bes
c=\sum_k(e_iR(a_{k1})\cdots R(a_{kp_k}))\cdot(e_j R(b_{k1}\cdots R(b_{kq_k}))
\ees
be such a nonnilpotent element. Consider the finite dimensional subalgebra $A'$ of $A$ generated by $e_i,e_j, a_{k\mu},b_{k\nu}$ for all $k,\mu,\nu$. Then $c\in A'\setminus Nil(A')$. By the Principal Wedderburn Theorem for Jordan algebras (see \cite{Jac}), there exists a decomposition $A'=A'_s+A'_n$, where $A'_s$ is semisimple and $A'_n$ is nilpotent; $e_i,e_j\in A'_s$. We have $(I_i\cap A'_s)\cdot (I_j\cap A'_s)\neq (0)$. Hence the problem has been reduced to a finite dimensional semisimple algebra $A'_s$, the idempotents $e_i,e_j$ are primitive idempotents. Here we can further assume that $e_i,e_j$ lie in the same simple component and, in fact, are included in a same frame. It is known \cite{Jac} that all idempotents in a frame of a finite dimensional simple Jordan algebra over an algebraically closed field are strongly connected.

\smallskip
Without loss of generality, we can assume that the algebra $A$  is not decomposed into a direct sum. Hence the idempotents $e_1,\ldots,e_r$ are pairwise strongly connected.  If $r=1$ then $A=F\cdot e_1$.  If $r\geq 2$ then then  by the Coordinatization Theorems of Jacobson and Osborn   \cite{Jac} $A$ is a Jordan algebra of a bilinear form or  $A$ is isomorphic to the algebra of $*$-hermitian $n\times n$ matrices over an associative  or alternative (for $n\leq 3$) algebra with involution $(B,*)$.   Since $\{e_i,A, e_i\}=F\cdot e_i$,  we have $H(B.*)=F$,  which implies that  $\dim_FA\leq \dim_FB<\infty$.

\ctd

\begin{cor}\label{cor0}
$H_{\0}/Nil\,H_{\0}\cong   A'\oplus A_1\oplus\cdots\oplus A_s$, where $\dim_F A'<\infty,\ A_i$'s are Jordan algebras of symmetric bilinear forms.
\end{cor}

\section{The structure of the heart $h(J)$ (radical part).}

\hspace{\parindent}

Recall that $X=X(J)$ is the $T$-ideal of the free Jordan algebra introduced in \cite{Zel5} and in the Introduction.  There exists $d\geq 1$ such that in an arbitrary Jordan algebra $A$ we have
\bes
R_A\la X(A)\ra=\sum_{i=1}^d R(X(A))^i.
\ees
It is straightforward that this assertion extends to superalgebras.

\begin{lem}\label{lem3.00}
Let $J$ be a primitive Jordan superalgebra with a modular ideal $K$ such that $X(J)\neq (0)$. Then 
\bes
R\la J\ra=(F\cdot Id+R_J\la K\ra) \sum_{i=0}^dR(X(J))^i.
\ees
\end{lem}
\prf
We have $X(J)^3\neq (0)$ and therefore 
\bes
K+X(J)^3=J.
\ees
Consider a multiplication operator $U=R(u_1)\cdots R(u_p),\, u_i\in J$. We assume that
 every multiplication operator of length $<p$ lies in the  right hand side of the equality in the lemma.
 
 Without loss of generality we assume that all $u_i\in K\cup X(J)^3$.  If $p\geq 3$ then by \eqref{D2} the operator $U$ is a linear combination of operators of length $<p$ and operators of length $p$ containing an inner derivation $D(u_i,u_j)=D$. If the derivation $D$ lies in $D(X(J)^3,X(J)^3)$ then by \eqref{D0} move it to the right side modulo shorter operators. If $D$ lies in $D(K,K)$ then similarly  move it to the left side modulo shorter operators. If $D$ lies in $D(K,X(J)^3)$ then move it to the right side and notice that by \eqref{D1} $D(K,X(J)^3)\subseteq D(X(J),X(J))$.
 
 \smallskip
 
 If $p=1$ then the assertion is obvious.  If $p=2$ then we need to consider only
 \bes
 R(X(J)^3R(K)&\subseteq& R(K)R(X(J)^3)+D(X(J)^3,K)\\
 &\subseteq& R(K)R(X(J))+R(X(J))^2.
 \ees
 This completes the proof of the lemma.
 
 \ctd
 
 Let $FJ$ be the free Jordan superalgebra of countable rank. 
 Denote by  $T_n$ be the ideal of identities satisfied by all algebras of the proposition \ref{prop2.4} not containing $n+1$ pairwise orthogonal idempotents. Clearly, $T_n\neq (0)$.

 Let $g=g(x_1,\ldots,x_m)$ be a multilinear element of degree $m$, and let $f\in g(S^{\la 2d+1\ra})$. Let 
 $$
 P=id_{FJ}(f)_{\0}\cap T_{2m+1}((S^{\la 2d+1\ra})_{\0})\cap (X(FJ))_{\0}.
 $$
  
 \begin{lem}\label{lem3.01}
 The ideal $P$ is a nil ideal.
 \end{lem}
 \prf
 Since $F$ is uncountable, it suffices to show that the ideal $P$ is quasi-invertible. Let $0\neq a\in P$ such that $(FJ)U(1-a)$ is a proper inner ideal. Then $(FJ)U(1-a)\subseteq K$ for a maximal
inner ideal $K$ in $FJ$.  Let $I(K)$ be the maximal ideal of $FJ$ contained in $K$. The superalgebra $J=FJ/I(K)$ is primitive.  Since $a\in S$  and $a\in X(FJ)$ it follows that $S(J)\neq (0)$ and $X(J)\neq (0)$. By Lemma \ref{lem3.00} the primitive superalgebra $J$ satisfies  the conditions \eqref{id000} of the section 3.  The heart ideal of the superalgebra $J$ is $\mathcal H=S^{\la 2d+1\ra}(J)$.
 
 Since $f\in g(S^{\la 2d+1\ra})$ we conclude that $g(\mathcal H)\neq (0)$. Hence by Lemma \ref{lem2.3} and Corollary \ref{cor0} the algebra $\mathcal H_{\0}/Nil(\mathcal H_{\0})$ satisfies $T_{2m+1}=(0)$.  Since $a\in T_{2m+1}(\mathcal H_{\0})$, it follows that $a\in Nil\,(\mathcal H_{\0})$, hence  $JU(1-a)=J$, a contradiction.
 
 This completes the proof of the lemma.
 
 \ctd


Recall that the {\em McCrimmon radical} $\mathcal M(A)$ of a Jordan algebra $A$ is defined as the smallest ideal $I$ of $A$ with the property that the factor-algebra $A/I$ is nondegenerate.  

\begin{lem}\label{lem2.6}
The ideal $P$ lies in the McCrimmon radical $\mathcal M(FJ_{\0})$ of $FJ_{\0}$.
\end{lem}
\prf
Let $a\in P$, and let $y$ be a free even variable that is not involved in $a$. The element $\{a,y,a\}$ is nilpotent. Hence the inner ideal $K=\{a,FJ_{\0},a\}$ is nil of bounded degree and by \cite[Lemma 17]{Zel1} $K=\mathcal M(K)$.  Now by \cite[Corollary of Lemma 15]{Zel1} we have $\{K,K,K\}=\{K,\mathcal M(K),K\}\subseteq \mathcal M(FJ_{\0})$.

Hence $a^6\in \mathcal M(FJ_{\0})$ and $P$ is nil of index 6 modulo $\mathcal M(FJ_{\0})$.  By \cite{Zel1}  again,  $P\subseteq \mathcal M(FJ_{\0})$.

\ctd
\begin{cor}\label{corPsubsetM}
$P(J_{\0})\subseteq {\cal {M}}(J_{\0})$.
\end{cor}

\begin{lem}\label{lem2.7}
$Nil(H_{\0})=\mathcal M(H_{\0})$.
\end{lem}
\prf
Recall that $H_{\0}=(S(J)^{\la 2d+1\ra})_{\0}$.  If $X(J_{\0})=(0)$ then $H_{\0}$ is a Jordan $PI$-algebra, and the result follows from \cite{Zel1,Zel4}.  Hence we may assume that $X(J_{\0})\neq (0)$ and then $X(J_{\0})\supseteq H_{\0}$.
Let $0\neq a\in Nil(H_{\0})$ be a value of an element $f\in S^{\la 2d+1\ra}$ of total degree $m$. By lemma \ref{lem2.6} we have 
$$
\{a,T_{2m+1}(H_{\0}),a\}\subseteq P(J_{\0})\subseteq H_{\0}\cap \mathcal M(J_{\0})=\mathcal M(H_{\0}).
$$
Also, $(Nil(H_{\0})+T_{2m+1}(H_{\0}))/T_{2m+1}(H_{\0})$ is a nil $PI$-algebra, hence it coincides with its  McCrimmon radical.

Consider an $m$-sequence in the algebra $H_{\0}$ that starts with $a$:
\bes
a_1=a,\, a_2=\{a,u_1,a\},\, a_3=\{a_2,u_2,a_2\},
\ees
and so on. Due to \cite{Zel4}, every $m$-sequence that starts with an element from McCrimmon radical vanishes. Hence some member $a_i$ lies in $T_{2m+1}(H_{\0})$, hence it lies in $P(J_{\0})$,  hence it lies in $\mathcal M(J_{\0})$. Now it follows that the $m$-sequence above vanishes.

\ctd
\smallskip

Let us fix a maximal system of pairwise orthogonal idempotents in $H_{\0}:\, e_1,\ldots,e_r$.  Let also $e_0=1-(e_1+\cdots +e_k)$. By {\em Pierce element} we mean an element in
\bes
(\sum_{i=0}^r\{e_i,J,e_i\})\cup (\bigcup_{0\leq i\neq j\leq r}\{e_i,J,e_j\}).
\ees
A product of two Pierce elements is a Pierce element as well. Observe also that if $a\in H$ then the Pierce elements $\{e_0,a,e_i\}, \, \{e_0,a,e_0\}\in H$.

As it was mentioned above, $\{e_i,H_{\0},e_i\}\subseteq F\cdot e_i+Nil(H_{\0})$ for any $i>0$. For an even element $a=\sum_{i=1}^r\a_i e_i+\sum_{i\neq j}\{e_i,a,e_j\}$ mod $Nil(H_{\0})$ we define $tr(a)=\sum_{i=1}^r \a_i$.

\smallskip
We call a system of operators $w_n=R(a_{n1})\cdots R(a_{nr_n})$, where $a_{ij}$'s are Pierce elements from $H_{\0}$, a {\em $w$-system}, if
\bes
w_{n+1}=w_nR(b_1)\cdots R(b_k)R(a_{nr_n}')\cdots R(a_{n1}'),
\ees
 where $0\leq k\leq 2$, $b_i$ are Pierce elements from $Nil(H_{\0})$, and for all $i,\,1\leq i\leq r_n,$ except, may be, one, $a_{ni}'=a_{ni}$. The element $a_{ni}'$ that is not equal to $a_{ni}$ is a Pierce element lying in $a_{ni}D(J_{\1},J_{\1})$.

A Pierce element $a\in I_{\0}$ is called an {\em $w$-element} if there exists $N=N(a)\geq 1$ such that for every $w$-sequence $w_1,w_2,\ldots$ that starts with $w_1=R(a),$ we have $w_{N(a)}=0$.

The following lemma was implicitly proved in \cite{Kap1} and \cite{MarZel}.
\begin{lem}\label{lem2.8}
Let $a\in H_{\0}$ be a Pierce $w$-element. Then $tr((a\cdot J_{\1})J_{\1})=(0)$.
\end{lem}
\prf
Choose Pierce elements $x,y\in J_{\1}$ such that $tr(ax\cdot y)\neq 0$. Since $ax\cdot y$ is a Pierce element then
\bes
ax\cdot y=\sum_{i=1}^r \a_i'e_i+b',\ b'\in Nil(H_{\0}),\ \sum_i\a_i'\neq 0.
\ees
More generally, let $w_1=R(a),\ldots,w_n$ be a $w$-sequence and $tr(xw_n\cdot y)\neq 0$.  If $y\in \sum_i\{e_i,J,e_i\}$ then we can assume that $y\in \{e_i,J,e_i\}$.  Arguing as above we get
\bes
xw_n\cdot y=\sum_{i=1}^n\a_ie_i+b,\ b\in Nil(H_{\0}),\  \sum_i\a_i\neq 0.
\ees
Multiplying both sides of this equality by $R(y)$, we get
\bes
x'w_n+xw_n'=\tilde\a y+by,
\ees
where $'=R(y)^2,\,\tilde\a\in F$.

We claim that $\tilde\a\neq 0$. Indeed, if $y=2\{e_i,y,e_j\},\, i\neq j$, then $xw_n\cdot y=\a_ie_i+\a_je_j+b,\, b\in Nil(H_{\0}),\, \a_i+\a_j\neq 0.$
Then $\tilde\a=\tfrac12(\a_i+\a_j)\neq 0$. If $y=\{e_i,y,e_i\}$ then $xw_n\cdot y=\a_ie_i+b,\, b\in Nil(H_{\0})$, and $\tilde\a=\a_i\neq 0$.

The operator $\tilde\a\cdot Id+R(b)$ is invertible, let
\bes
(\tilde\a\cdot Id +R(b))^{-1}=\sum\a_{ij}R(b^i)R(b^j),
\ees
where $R(b^0)=Id$. Now
\bes
y=\sum_{i,j}\a_{ij}(x'w_n+xw_n')R(b^i)R(b^j).
\ees
For arbitrary elements $u,v\in J_{\1},\, c\in J_{\0}$ we have $tr(cD(u,v))=0$, hence $tr(u\cdot vc)=tr(v\cdot cu)$. Let $w_n=R(a_{n1})\cdots R(a_{nr_n}),\, a_{ij}$ are Pierce elements from $H_{\0}$.
Suppose that $tr(xw_n\cdot x'w_nR(b^i)R(b^j))\neq 0$.  Using the remark above we get
\bes
tr(xw_n\cdot x'w_nR(b^i)R(b^j))=tr(xw_nR(b^j)R(b^i)R(a_{nr_n})\cdots R(a_{n1})\cdot x')\neq 0.
\ees
Let $w_{n+1}=w_nR(b^j)R(b^i)R(a_{nr_n})\cdots R(a_{n1})$,  then  we have $w_{n+1}\neq 0$.
 Now let $tr(xw_n\cdot xw_n'R(b^i)R(b^j))\neq 0$.  We have
\bes
w_n'=\sum_{k=1}^{r_n}R(a_{n1})\cdots R(a_{nk}R(y)^2)\cdots R(a_{nr_n}).
\ees
Hence there exists $k,\, 1\leq k\leq r_n,$ such that
\bes
tr(xw_nR(b^j)R(b^i)R(a_{nr_n})\cdots R(a_{nk}R(y)^2)\cdots R(a_{n1})\cdot x)\neq 0.
\ees
Let $w_{n+1}=w_{n}R(b^j)R(b^i)R(a_{nr_n})\cdots R(a_{nk}R(y)^2)\cdots R(a_{n1})$, then again $w_{n+1}\neq 0.$  Since $a$ is a $w$-element, we have got a contradiction.

\ctd

Recall that an ideal $I$  of an algebra $A$ is called {\em strongly nilpotent of index $k$} if  any product of elements of $A$ that involves at least $k$ factors from $I$ is zero.

\begin{lem}\label{lem2.9}
Let $A$ be a Jordan algebra and $I\idl A$. If $I^{2k+1}=0$ then $I^3$ is strongly nilpotent of index $\leq k+1$.
\end{lem}
\prf
It suffices to prove that any product in $A$ that involves $k+1$ factors from $I^3$ lies in $I^{\la 2k+1\ra}$.  
This evidently holds for $k=1$.  Let $k>1$ and assume that the statement is true for the smallest values of $k$.
Let $w$ be a product that involves  $k+1$ factors from $I^3$. We write $w$ in the form $w=iW$, where $i\in I^3$ and $W=R(a_1)R(a_2)\cdots R(a_n)$ where the subproducts $a_1,\ldots,a_n$ all together  contain $k$ elements from $I^3$. By the Jordan operator identity \eqref{OSJ} we may assume that all the subproducts $a_i$ have at most two factors. Let $m$ be the biggest number for which $a_m$  contains factors from $I^3$,
then either $a_m=i'a,\, i'\in I^3,\,a\in A^{\sharp},$ or $a_m\in (I^3)^2$. In the first case the product    $w'=iR(a_1)R(a_2)\cdots R(a_{m-1})$ contains at least $k$ factors from $I^3$ and by induction $w'\in I^{\la 2(k-1)+1\ra}$. Then $w'R(a_m)\in I^{\la 2k-1\ra)}I^3\subseteq I^{\la 2k+1\ra}$ and $w\in I^{\la 2k+1\ra}$. In the second case, by the Jordan operator identity we have $R(a_m)\in (R(I^2))^3+\sum_{k=4}^6 (R(I))^k$.
In this case the product  $w'=iR(a_1)R(a_2)\cdots R(a_{m-1})$ contains $k-1$ factors from $I^3$ and by induction $w'\in I^{\la 2(k-2)+1\ra}=I^{\la 2k-3\ra}$. Then 
\bes
w'R(a_m)&\in& (I^{\la 2k-3\ra} I^2)I^2 +(((I^{\la 2k-1\ra}I)I)I)I\\
& \subseteq& I^{\la 2k-1\ra}I^2+(I^{\la 2k-1\ra}I)I\subseteq I^{\la 2k+1\ra}.
\ees
\ctd

\begin{lem}\label{lem2.10}
Let $L$ be a nilpotent ideal of $H_{\0},\ L^{2k+1}=(0)$. Let $w_n=R(a_{n1})\cdots R(a_{n{r_n}})$ be a $w$-sequence.  Suppose that at least one of the elements $a_{11},\ldots,a_{1r_1}$ lies in $(L^3)^3$. Then $w_{k+1}=0$.
\end{lem}
\prf
We claim that for any $n\geq 1$ at least one of the elements $a_{n1},\cdots,a_{nr_n}$ lies in $(L^3)^3$ and at least $n-1$ other elements lie in $L^3$. This assertion is true for $n=1$. Let $n=2$. Then
\bes
w_2=w_1R(b_1)\cdots R(b_k)R(a_{1r_1}')\cdots R(a_{11}')
\ees
 and it remains to notice that $(L^3)^3D(J_{\1},J_{\1})\subseteq L^3$.

Let $n\geq 3, w_n=w_{n-1}R(b_1)\cdots R(b_k)R(a_{n-1,r_{n-1}}')\cdots R(a_{n-1,1}')$. There are $\geq n-1$ elements from $L^3$ among $a_{n-1,1},\cdots,a_{n-1,r_{n-1}}$. At most one of them was affected by a derivation from $D(J_{\1},J_{\1})$. Hence $w_n$ involves an element from $(L^3)^3$ (in $w_{n-1}$) and $\geq 2(n-2)\geq n-1$ elements from $L^3$.

If $L^{2k+1}=(0)$ then by Lemma \ref{lem2.9} the ideal $L^3$ is strongly nilpotent of degree $\leq k+1$, which implies $w_{k+1}=0$.

\ctd

\begin{cor}\label{cor2.1}
Let $L$ be a solvable ideal of $H_{\0}$. Then
\bes
tr(((L^3)^3)^3\cdot J_{\1})J_{\1})=(0).
\ees
\end{cor}
\prf
For a solvable ideal $L\idl H_{\0}$ the ideal $L'=L^3$ is nilpotent \cite{MedZel}.
By lemma \ref{lem2.10} every Pierce element from $((L')^3)^3$ is a $w$-element. Now the assertion follows from lemma \ref{lem2.8}.

\ctd

\begin{lem}\label{lem2.10'}
Let $J=J_{\0}+J_{\1}$ be a Jordan superalgebra,  $I\idl J,\ P\idl I_{\0}$.   Then
\bes
P^{\la 2k+1\ra}\cdot J_{\0}&\subseteq& P^{\la 2k-1\ra},\\
P^{\la 2k+1\ra}D(J_{\1},J_{\1})&\subseteq& P^{\la 2k-1\ra},\\
(P^{\la 2k+1\ra}J_{\1})J_{\0}&\subseteq &(P^{\la 2k-3\ra}J_{\1})P+P^{\la 2k-3\ra}J_{\1}.
\ees
\end{lem}
\prf
We have $P^3\cdot J_{\0}\subseteq (PJ_{\0})P^2+(PJ_{\0}\cdot P)P\subseteq I_{\0}\cdot P\subseteq P$, which gives a base for induction for the first inclusion.
Now by induction, for any $a\in P^{\la 2k-1\ra};\, p,q\in P,\,x\in J_{\0}$ we have
\bes
(a\cdot pq)x&=&(ax)(pq)+(pq,a,x)=(ax)(pq)+(p,a,qx)+(q,a,px)\\
&\in& P^{\la 2k-3\ra}\cdot P^2+P^{\la 2k-1\ra}\cdot I_{\0}\subseteq P^{\la 2k-1\ra},\\
((a\cdot p)q)x&=&(ap)(qx)+(ap,q,x)=(ap)(qx)+(a,q,px)+(p,q,ax)\\
&\in& P^{\la 2k-1\ra}\cdot I_{\0}+P^{\la 2k-3\ra}\cdot P^2+(P^{\la 2k-3\ra}\cdot P)P\subseteq P^{\la 2k-1\ra}.
\ees
This proves the first inclusion.

Now, let $x,y\in J_{\1}, \, D=D(x,y)$, then we have
\bes
D(P)\subseteq I_{\0},\, D(P^3)\subseteq D(P)P^2+(D(P)P)P\subseteq P,
\ees
which gives a base for induction for the second inclusion.  Furthermore, for any  $a\in P^{\la 2k-1\ra},\, p,q\in P, $ we have
\bes
D(a\cdot p^2)&=&D(a)\cdot p^2+2a(p\cdot D(p))\in P^{\la 2k-3\ra}\cdot P^2+ P^{\la 2k-1\ra}P\subseteq P^{\la 2k-1\ra},\\
D(ap\cdot q)&=&(D(a)p)q+(aD(p))q+(ap)D(q)\in (P^{\la 2k-3\ra} P)P+P^{\la 2k-1\ra} I\\
&\subseteq& P^{\la 2k-1\ra}.
\ees
This proves the second inclusion.

Finally,  we have
\bes
(P^{\la 2k+1\ra}J_{\1})J_{\0}&\subseteq& (P^{\la 2k+1\ra},J_{\1},J_{\0})+P^{\la 2k+1\ra}J_{\1}.
\ees
Consider
\bes
(P^{\la 2k+1\ra},J_{\1},J_{\0})&\subseteq&(P^{\la 2k-1\ra}\cdot P,J_{\1},J_{\0})\subseteq (P^{\la 2k-1\ra},J_{\1},I_{\0})+(P,J_{\1},P^{\la 2k-3\ra})\\
&\subseteq &(P^{\la 2k-3\ra}\cdot P,J_{\1},I_{\0})+(P,J_{\1},P^{\la 2k-3\ra})\subseteq (P^{\la 2k-3\ra},J_{\1},P).
\ees
Since $P^{\la 2k+1\ra}J_{\1}\subseteq P^{\la 2k-3\ra}J_{\1}$, this proves the lemma.

\ctd

Let $P\subseteq Nil(H_{\0})$ be an ideal of $H_{\0}$ that is spanned by $w$-elements.
\begin{lem}\label{lem2.11}
Let $u\in \underbrace{R(J)\cdots R(J)}_{\text{k}}$ be an even multiplication operator of length $k\geq 2$. Then $tr(P^{\la 4k-3\ra}u)=(0)$.
\end{lem}
\prf
For $k=2$ the assertion follows from lemma \ref{lem2.8}.  Suppose that the assertion is not true for $u=R(a_1)\cdots R(a_k)$ and $k$ is minimal with this property.

If $a_1\in J_0$ then by lemma \ref{lem2.10'} $P^{\la 4k-3\ra}\cdot a_1\subseteq P^{\la 4k-5\ra}\subseteq P^{\la 4(k-1)-3\ra}$,  which contradicts the induction assumption on $k-1$.  Hence $a_1\in J_{\1}$.

If there are 3 consecutive odd elements in $a_1,\ldots,a_k$ then by \eqref{D2} $u$ is a linear combination of operators of length $k$ containing $D(J_{\1},J_{\1})$ and of operators of length $\leq k-1$.    Move $D(J_{\1},J_{\1})$ by \eqref{D0} to the left, then again by lemma \ref{lem2.10'} $P^{\la 4k-3\ra}D(J_{\1},J_{\1})\subseteq P^{\la 4k-5\ra}$ and we can again use the induction assumption.

If there are 2 consecutive even elements then using the Jordan operator  identity \eqref{OSJ}  move them to the left end which leads to the earlier case $a_1\in J_{\0}$.

If there are 2 consecutive odd elements then using \eqref{OSJ} move them to the right end.

So, we have 3 possibilities:
\begin{itemize}
\item
$k=2q$ is even, $a_{2i}\in J_{\0},\, a_{2i-1}\in J_{\1},\ 1\leq i\leq q$,
\item
$k=2q, \,a_{2i}\in J_{\0},\, a_{2i-1}\in J_{\1},\ 1\leq i\leq q-1,$ and $a_{2q-1},a_{2q}\in J_{\1}$,
\item
$k=2q+1, \, a_{2i}\in J_{\0},\, a_{2i+1}\in J_{\1},\ 0\leq i\leq q$.
\end{itemize}

Now,  by lemma \ref{lem2.10'}
$$
P^{\la 4k-3\ra}R(a_1)R(a_2)\subseteq P^{\la 4k-7\ra}(R(a_1)R(P)+R(J_{\1})).
$$
By induction,
\bes
tr(P^{\la 4k-7\ra}R(J_{\1})R(a_3)\cdots R(a_k))=tr(P^{\la 4(k-1)-3\ra}R(J_{\1})R(a_3)\cdots R(a_k))=0.
\ees
On the other hand,  by the Jordan operator identity \eqref{OSJ}, the operators of type  $R(a_1)R(P)R(a_3)\cdots R(a_k)$ are skew-symmetric in all even factors modulo operators of length $\leq k-1$, hence we can move $R(P)$ to the right as far as possible.
Let
\bes
V&=&P^{\la 4k-7\ra }R(a_1)R(P)R(a_3)\cdots R(a_k), \\
 W&=&P^{\la 4k-7\ra}\hbox{ (operators of length $\leq k-1$)},
 \ees
 then  according  to the three possibilities above we have the three cases
$$
V\subseteq P+W, \ V\subseteq (PJ_{\1})J_{\1}+W,\ V\subseteq (J_{\1}P)J_{\1}+W.
$$
It remains to use lemma \ref{lem2.8} and the induction assumption.

 \ctd

 \begin{lem}\label{lem2.12}
 Let $P\subseteq Nil(H_{\0})$ be an ideal of $H_{\0}$ spanned by $w$-elements. Then $P^{\la 4d-3\ra}=(0)$.
 \end{lem}
 \prf
 By lemma \ref{lem2.11} for an arbitrary multiplication even operator $w$ of length $d$ we have
 $tr(P^{\la 4d-3\ra}w)=(0)$. Since the superalgebra $J$ has multiplicative length $d$, it follows that $tr(id_J\la P^{\la 4d-3\ra}\ra)=(0)$.  But $id_J\la P^{\la 4d-3\ra}\ra$, if nonzero, contains $H_{\0}$,  a contradiction.

 \ctd

 Now  let $L$ be an arbitrary solvable ideal of $H_{\0}$.   It was noted in  the proof of Corollary \ref{cor2.1}  that every Pierce element from $((L^3)^3)^3$ is a $w$-element. Hence by  lemma \ref{lem2.12} we have $(((L^3)^3)^3)^{\la 4d-3\ra}=(0)$, and then  by lemma \ref{lem2.2'} 
 $$
 (((L^3)^3)^3)^{8d}\subseteq (((L^3)^3)^3)^{\la 4d-3\ra}=(0).
 $$
 Therefore,  $H_{\0}$ contains a largest solvable ideal.  From now on let $L$ denote the largest solvable ideal of $H_{\0}$.

 \smallskip

 Now our aim is to prove that $L=\mathcal M(H_{\0})=Nil(H_{\0})$, that is, the factor-algebra $H_{\0}/L$ is nondegenerate.

Let $J$ be an arbitrary Jordan algebra.  Recall that an element $b\in J$ is called {\em an absolute zero divisor of rank $n$} if for any $a_1,\ldots,a_{2n+1}\in \hat J$ we have
 \bes
 bV(a_1,a_2)V(a_3,a_4)\cdots V(a_{2n-1},a_{2n})V(a_{2n+1},b)=0,
 \ees
 where $xV(y,z)=\{x,y,z\}$.
 An ordinary absolute zero divisor is a zero divisor of zero rank. In \cite{Zel1} it was proved that a degenerate Jordan algebra contains nonzero absolute zero divisors of rank $n$ for any $n\geq 0$.

In what follows $J^{\sharp}$ means the {\em  unital hull} of an algebra $J$; that is,  $J^{\sharp}=J$ if $J$ has the unit element, and $J^{\sharp}=F\cdot  1+J$ if $J$ has no unit element.
 \begin{lem}\label{lem2.13}
If $b$ is an absolute zero divisor of rank $n$ then for an arbitrary element $x\in J^{\sharp}$ the element $bU(x)$ is also an absolute zero divisor of rank $n$.
\end{lem}\prf
Notice that the following identity holds in $J^{\sharp}$ for an invertible element $x$:
\bes
V(a_1,a_2)U(x)=U(x)V(a_1U(x),a_2U(x^{-1}).
\ees
One can easily check that it holds in special algebras, and since its degree is 5, it holds in arbitrary Jordan algebras \cite{Glennie}.   Therefore,  if $x$ is invertible then
\bes
&bU(x)V(a_1,a_2)\cdots V(a_{2n-1},a_{2n})V(a_{2n+1},bU(x))=&\\
& bV(a_1U(x^{-1}),a_2U(x))\cdots V(a_{2n-1}U(x^{-1}),a_{2n}U(x))V(a_{2n+1}U(x^{-1}),b)U(x)=0.&
\ees
Consider the algebra of infinite series $J^{\sharp}[[t]]$. The element $b$ is an absolute zero divisor of rank $n$ in $J^{\sharp}[[t]]$. The element $1+xt$ is invertible in $J^{\sharp}[[t]]$, hence for any
$a_1,\ldots,a_{2n+1}\in J^{\sharp}$
\bes
bU(1+xt)V(a_1,a_2)\cdots V(a_{2n-1},a_{2n})V(a_{2n+1},bU(1+tx))=0.
\ees
The coefficient at $t^4$ is
\bes
bU(x)V(a_1,a_2)\cdots V(a_{2n-1},a_{2n})V(a_{2n+1},bU(x))=0.
\ees
\ctd
\begin{cor}\label{cor2.2}
The span of the set of all zero divisors of rank $n$ is an ideal of $J$.
\end{cor}
It was shown in \cite{MedZel} that there exists $N\geq 1$ such that for any Jordan algebra and any ideal $I\idl J$
\bes
(R((I^3)^3)R(J))^N\subseteq R(J)R((((I^{10})^3)^3)^3)R(J).
\ees
Let $r$ be an even number such that $2^r\geq N$. Let $b_{ij}\in H_{\0}$ be absolute zero divisors of rank $2^r(9(r+1)+2)$ modulo $L$; $1\leq i\leq r+1,\ 1\leq j\leq 9$.  
Let $c_i=(((b_{i1}b_{i2})b_{i3})((b_{i4}b_{i5})b_{i6})((b_{i7}b_{i8})b_{i9}),\ 1\leq i\leq r+1$, and let $c$ be a product $c_1\cdots c_{r+1}$ with arbitrary brackets.  By Lemma \ref{lem2.13} for an  absolute zero divisor $b$ of rank $n$ its Pierce coordinates $\{e_i,b,e_j\}$ are absolute zero divisors of rank $n$ as well.  Hence without loss of generality we may assume that all $b_{ij}, c_i$ and $c$ are Pierce elements. 

\begin{lem}\label{lem2.14}
For an arbitrary $w$-sequence $w_1=R(c),w_2,\ldots,w_n,\ldots $ we have $w_{r+1}\in R(H_{\0})R(((L^3)^3)^3)R(H_{\0})$.
\end{lem}
\prf
Observe that every operator $w_n$ contains $2^{n-1}$ multiplication operators
\bes
w_n=\cdots R(cD_{11}\cdots D_{1\,\mu_1})\cdots R(cD_{2^{n-1},1}\cdots D_{2^{n-1},\,\mu_{2^{n-1}}})\cdots,
\ees
where $D_{ij}\in D(J_{\1},J_{\1}),\ \mu_1+\cdots+\mu_{2^{n-1}}\leq n-1$. Consider the operator $w_{r+1}$. Each element $cD_{j1}\cdots D_{j\,\mu_j}$ in $w_{r+1}$ can be represented as
$cD_{j1}\cdots D_{j\,\mu_j}=\sum_{\nu}c_{j\nu}$, where $c_{i\nu}$ is obtained from $c=c_1\cdots c_{r+1}$ by application of derivations $D_{j1},\ldots,D_{j\,\mu_j}$ to the elements $c_1,\ldots,c_{r+1}$.
Since $\mu_1+\cdots+\mu_{2^r}\leq r$, it follows that in each summand $A=\cdots R(c_{1\nu_1})\cdots R(c_{2^r\nu_{2^r}})\cdots$ at least one element $c_i$ stays untouched in all products $c_{1\nu_1},\ldots,c_{2^r\nu_{2^r}}$.

Consider the ideal $Q$ of $H_{\0}$ generated by $b_{i1},\ldots,b_{i9}$. Then $c_i\in (Q^3)^3$. Since the element $c_i$ occurs $2^r\geq N$ times in $A$, it follows that
\bes
A\in (R(H_{\0})R((Q^3)^3)R(H_{\0}))^N\subseteq R(H_{\0})R((((Q^{10})^3)^3)^3)R(H_{\0}).
\ees
Every product from $Q^{10}$ contains some element $b_{ij},\ 1\leq j\leq 9,$ at least twice. The length of this product is less then total degree of $w_{r+1}$.  Let $d(w_j)$ be a maximal total degree of $w_j$. Then $d(w_1)=9(r+1),\, d(w_{j+1})\leq 2d(w_j)+2$, which implies $d(w_{r+1})\leq 2^r(9(r+1)+2)$. Since all $b_{ij}$'s are absolute zero divisors of rank $2^r(9(r+1)+2)$ modulo $L$, it follows that the arising products from $Q^{10}$ lie in $L$. Hence $A\in R(H_{\0})R(((L^3)^3)^3)R(H_{\0})$.  Since $A$ is an arbitrary summand of $w_{r+1}$, we conclude that $w_{r+1}\in  R(H_{\0})R(((L^3)^3)^3)R(H_{\0})$.

\ctd

Let $\tilde L$ be the span of all absolute zero divisors of rank $2^r(9(r+1)+2)$ of $H_{\0}$ modulo $L$.

\begin{lem}\label{lem2.15}
$\tilde L=L$.
\end{lem}
\prf
If an element $a$ is an absolute zero divisor of rank $2^r(9(r+1)+2)$ then by Corollary \ref{cor2.2} every Pierce element $\{e_i,a,e_j\}$ is an absolute zero divisor of rank $2^r(9(r+1)+2)$.

In the proof of lemma \ref{lem2.8} we showed that if $a$ is a Pierce element and $tr((J_{\1}a)J_{\1})\neq (0)$ then there exists a $w$-sequence $w_1=R(a),w_2,\ldots$ such that for any $i\geq 1$ we have $tr(J_{\1}w_i\cdot J_{\1})\neq (0)$.
Let $b_{ij},\,1\leq i\leq r+1,\,1\leq j\leq 9$, be absolute zero divisors of rank $2^r(9(r+1)+2)$; let $b_{ij}'$ be their Pierce components, let
\bes
c_i'=(((b_{i1}'b_{i2}')b_{i3}')((b_{i4}'b_{i5}')b_{i6}'))((b_{i7}'b_{i8}')b_{i9}')
\ees
and let $c'=c_1'\cdots c_{r+1}'$ with arbitrary brackets. If $w_1=R(c'),\,w_2,\ldots $ is a $w$-sequence then, by Corollary  \ref{cor2.1} and Lemma \ref{lem2.14}  $tr(J_{\1}w_{r+1}\cdot J_{\1})=(0)$. This implies $tr(J_{\1}c'\cdot J_{\1})=(0)$
and therefore $tr(J_{\1}c\cdot J_{\1})=(0), \ tr((J_{\1}((\tilde L)^3)^3)^{r+1})J_{\1})=(0)$.

In the proof of lemma \ref{lem2.12} we showed that if $P$ is an ideal of $H_{\0}$ lying in  $Nil(H_{\0})$ and $tr((J_{\1}\cdot P)J_{\1})=(0)$ then $P^{\la 4d-3\ra}=(0)$. Hence $(((\tilde L^3)^3)^{r+1})^{\la 4d-3\ra}=(0)$, the ideal $\tilde L$ is solvable and therefore $\tilde L=L$.
\ctd

\smallskip

We proved that the ideal $L=Nil(H_{\0})=\mathcal M(H_{\0})$ is solvable.

\section{The superalgebra $J$ coincides with its heart $J=h(J)=H$.}

Denote $L=Nil(H_{\0})$.
 \begin{lem}\label{lem2.16}
 $tr((L\cdot J_{\1})J_{\1})=(0)$.
 \end{lem}
 \prf
 In \cite {MedZel} it was shown that if $J$ is a solvable Jordan algebra  then  the ideal generated by $R(J^2)$ in the algebra of multiplications  $R\la J\ra$ is nilpotent.  Since the algebra $L$ is solvable, there exists $r\geq 1$ such that $(R(L)R(L^2))^r=0$.

\smallskip

 Let $a_1,a_2,\ldots\in L$. If $a_1,a_3,a_5,\ldots,a_{2k+1},\ldots$ are not all different then $R(a_1)R(a_2)\cdots \in R(L)R(L^2)$. 
 \smallskip
 
 Let $a,b\in L$. Then by \eqref{OSJ}
 \bes
 R(L)(FR(a)R(b)+FR(b)R(a))^3\subseteq R(L)R(L^2).
 \ees
 Hence
 \bes
 (R(L)(FR(a)R(b)+FR(b)R(a))^{3r}=(0).
 \ees
 If $tr((LJ_{\1})J_{\1})\neq (0)$ then there exists an $w$-sequence starting with $R(a),\, a\in L$, such that $tr((J_{\1}w_n)J_{\1})\neq (0)$ for any $n$. We have $w_3=R(a_{31})\cdots R(a_{3r_3})$, $a_{ij}\in L,\ r_3\geq 4$.
Changing notation: $w_3=R(a_1)R(b_1)R(a_2)R(b_2)\cdots$. The operator $w_{3+\nu}$ involves $\geq 2^{\nu}+1$ pairs $R(a_i)R(b_i)$. If $2^{\nu}+1\geq 2(3r-1)$ then some pair $R(a_i)R(b_i)$ or $R(b_i)R(a_i)$ occurs $\geq 3r$ times, which implies $w_{3+\nu}=0$.

\ctd

\begin{lem}\label{lem2.17}
$(L\cdot J_{\1})J_{\1}\subseteq L$.
\end{lem}\prf
The trace form is nondegenerate on $H_{\0}/L$. Hence it is sufficient to prove that for any $x,y\in J_{\1},\ \forall a\in H_{\0}\setminus L$ we have
\bes
tr(LR(x)R(y)R(a))=0.
\ees
But
\bes
R(x)R(y)R(a)&=&-R(xa\cdot y)+R(a)R(y)R(x)\\
&+&R(x)R(ya)-R(y)R(xa)+R(a)R(xy),
\ees
which by lemma \ref{lem2.16} completes the proof.

\ctd

By lemma \ref{lem2.2} $H_{\0}\neq L.$ Let $0\neq a\in J_{\0}\cup J_{\1}$ be an arbitrary nonzero element. Then $id_J\la a\ra\supset H$. Hence there exists an operator $w=R(a_1)\cdots R(a_n)$ such that $aw\in H_{\0}\setminus L$. We call $w$ a {\em rescuing operator} for $a$ if

1) $n$ is a minimal with this property,

2) all elements $a_i$ lie in $J_{\0}\cup J_{\1}$ and are Pierce elements with respect to $e$ ($e$ is an identity of $H_{\0}\ mod\ L)$.

\begin{lem}\label{lem2.18}
Let $w$ be a rescuing operator. Then
\begin{enumerate}
\item
$a_n$ is odd,
\item
the sequence $a_1,\ldots,a_n$ does not contain two consecutive elements of the same parity,
\item
all even elements are skew-symmetric $mod\ L$ and all odd elements are symmetric $mod\ L$,
\item
if $w$ contains an even element $a_i=u^2$ then for any $j\neq i$ we have the equality 
\bes
&&aR(a_1)\cdots R(u^2)\cdots R(a_j)\cdots R(a_n)\\
&=&2aR(a_1)\cdots R(u)\cdots R(ua_j)\cdots R(a_n)\ (mod\ L).
\ees
\end{enumerate}
\end{lem}
\prf
1. If $a_n$ is even, consider the element $b=aR(a_1)\cdots R(a_{n-1})$. Then $b\in J_{\0},\,b\not\in L,\, (ba_n)U(e)=ba_n\ (mod\ L)$. If $b\not\in H$, then $b\in\{1-e,J_{\0},1-e\}$. But in this case the equality $(ba_n)U(e)=ba_n\ (mod\ L)$ is impossible for any Pierce element $a_n\in J_{\0}$. Hence $b\in H$ and $w'=R(a_1)\cdots R(a_{n-1})$ is already a rescuing operator, contradicting minimality of $w$.
\smallskip

2. By the Jordan operator identity \eqref{OSJ} we may assume that the two consecutive elements of the same parity  are $a_{n-1},a_n$. By the part 1, they can not be even. If they are odd, consider the element $c=aR(a_1)\cdots R(a_{n-2})$. We have $c\in J_{\0}$ and $c\notin L$ by lemma \ref{lem2.17}. If $c\in H$ then the operator $w'=R(a_1)\cdots R(a_{n-2})$ would be  rescuing, a contradiction. Hence $c\in\{1-e,J_{\0},1-e\}$ and both $a_{n-1}, a_n\in \{e,J_{\1},1-e\}$. By \eqref{OSJ} again,  the elements $a_{n-2}$ and $a_n$ can be permuted, hence by the previous arguments $a_{n-2}\in\{e,J_{\1},1-e\}\subseteq H$. 
Hence $c\in H$, a contradiction.   
\smallskip

3.  It follows from part 2 and the Jordan operator identity  \eqref{OSJ}. 
\smallskip

4. By part 3 we can assume that $j=i+1$ if $a_j$ is odd and $j=i+2$ if $a_j$ is even. In the first case the part 2 and \eqref{D1} imply
\bes
aw&=&a\cdots R(u^2)R(a_j)\cdots\equiv a\cdots (D(u^2,a_j)+R(a_j)R(u^2))\cdots \\
&\equiv&a\cdots 2D(u,ua_j)\cdots \equiv 2a\cdots R(u)R(ua_j)\cdots  \ (mod\ L).
\ees
In the second case $a_{i+1}=x$ is odd, and we have again by \eqref{D1}
\bes
aw&=&a\cdots R(u^2)R(x)R(a_j)\cdots\equiv 2a\cdots R(u)R(ux)R(a_j)\cdots \\
&\equiv &2a\cdots R(u)(D(ux,a_j)+R(a_j)R(ux))\cdots\\
& \equiv& 2a\cdots R(u)(D(u,xa_j)+D(x,ua_j))\cdots\\
& \equiv& 2a\cdots (-R(u)R(xa_j)R(u)+R(u)R(x)R(ua_j))\cdots\\
& \equiv& 2a\cdots R(u)R(x)R(ua_j)\cdots  \ (mod\ L).
\ees

\ctd

\begin{lem}\label{lem2.19}
For any $ 0\neq a\in J_{\0}\cup J_{\1}$ there exists a rescuing operator $w=R(a_1)\cdots R(a_n)$ such that all $a_1,\ldots,a_n\in \{e,J,e\}\cup\{e,J,1-e\}$.
\end{lem}
\prf
Let $w=R(a_1)\cdots R(a_n)$ be a rescuing operator for $a$, all elements $a_1,\ldots,a_n$ lie in $\{e,J,e\}\cup\{e,J,1-e\}\cup\{1-e,J,1-e\}$. Since $aw\in H_{\0}\setminus L$, it follows that $aw=awU(e)\ mod\ L$. If one of the odd elements in $a_1,\ldots,a_n$ lies in $\{1-e,J,1-e\}$ then by lemma \ref{lem2.18} we can assume that $a_n\in \{1-e,J,1-e\}$; but then $awU(e)=0$, a contradiction.
Let $a_i$ be even and $a_i\in\{1-e,J,1-e\}$. By lemma \ref{lem2.18} $i<n$. Hence by lemma \ref{lem2.18} the element $a_{i+1}$ is odd and therefore $a_{i+1}=a_{i+1}\cdot e$ or $a_{i+1}=2a_{i+1}\cdot e$.
Let $w'=R(a_1)\cdots R(a_{i-1}),\ w''=R(a_{i+2})\cdots R(a_n)$. We have $aw=aw'D(a_i,a_{i+1})w''\ mod\ L$. By \eqref{D1}, 
$D(a_i,a_{i+1}\cdot e)=D(a_i\cdot a_{i+1},e)$ since $a_i\cdot e=0$. Therefore, $aw=\pm aw'R(e)R(a_i\cdot a_{i+1})w''$. The number of elements from $\{1-e,J_{\0},1-e\}$ went down by 1.

\ctd
\begin{lem}\label{lem2.20}
$\{1-e,J,1-e\}=(0)$.
\end{lem}
\prf
We will show that a nonzero element from $\{1-e,J,1-e\}$ can not be rescued. Let $0\neq a\in \{1-e,J,1-e\}$, let $w=R(a_1)\cdots R(a_n)$ be a rescuing operator, $a_i\in \{e,J,e\}\cup\{e,J,1-e\}$.

First we notice that $n>1$. Indeed, if $w=R(a_1),\ a_1\in J_{\1}$, then $a_1\in \{e, J_{\1},1-e\}$ and $aR(a_1)\in \{e,J_{\0},1-e\}\subseteq L$, a contradiction.

If $a_1$ is even then  $a_1\notin L$, hence $a_1=\{e,a_1,e\}$ and therefore $aa_1=0$.

Let $a_1$ be odd, $a_2$ be even. Then again $a_2=\{e,a_2,e\},\ D(a,a_2)=0,$ hence $(aa_1)a_2=a(a_1a_2)$. This contradicts minimality of $n$.

\ctd

\begin{cor}\label{cor2.3}
$J=h(J)=H$ is a simple unital superalgebra.
\end{cor}

We will need  more information on the ideal $L=Nil\,(J_{\0})=M(J_{\0})$.  We know from section 4 that it is solvable.   Now we want to prove  that $L^{4d+1}=(0)$.

Let $V\subseteq J$ be an $J_{\0}$-submodule of $J$.  Define 
\bes
(V,L)_0=V, \ (V,L)_1=(VL)L+V\cdot L^2,
\ees
and for $i\geq 1$ 
\bes
(V,L)_{i+1}=((V,L)_{i}L)L+(V,L)_{i}\cdot L^2.
\ees
\begin{lem}\label{lem2.21}
(a) $(V,L)_n$ is a $J_{\0}$-submodule of $J$ for any $n$;

(b) $((V,L)_{n}\cdot J_{\1})J_{\0}\subseteq (V,L)_{n-1}\cdot J_{\1}+((V,L)_{n-1}J_{\1})\cdot L.$
\end{lem}
\prf
(a) Since $V=(V,L)_0$ is $J_{\0}$-invariant, we may use induction on $n$.  Let $v\in (V,L)_n;\ a,b\in L; x\in J_{\0}$. Modulo $(V,L)_{n+1}$ the following comparisons hold
\bes
((va)b)\cdot x&\equiv& (va,b,x)=(v,b,ax)+(a,b,vx)\equiv 0,\\
(v(ab))\cdot x&=&(v,a,b)x+((va)b)x\equiv (v,a,b)x\\
&=&(v,ax,b)-(v,x,b)a\equiv 0,
\ees
which proves (a).
\smallskip

(b) Notice that $(V,L)_n\subseteq (V,L)_{n-1}\cdot L$.  Let $v\in (V,L)_{n-1},\, l\in L,\, a\in J_{\0},\, x\in J_{\1}$. We have
\bes
(vl)x\cdot a&=&(vl)(xa)+(vl,x,a)=(vl)(xa)+(v,x,al)+(l,x,va)\\
&\in& (V,L)_{n-1}\cdot J_{\1}
+(V,L)_{n-1}\cdot J_{\1}\cdot L.
\ees
This completes the proof of the lemma.

\ctd

\begin{lem}\label{lem2.22}
$(J,L)_{d}=(0)$.
\end{lem}
\prf
 Denote by $f(n)$ the minimal length of resquing operators for nonzero elements from $(J,L)_n$.  Since $(J,L)_n$ is $J_{\0}$-invariant, it is clear that any minimal resquing operator $w=R(a_1)\ldots R(a_{f(n)})$ for $(J,L)_n$ should have $a_1\in J_{\1}$.  By the properties of a resquing operator from Lemma \ref{lem2.18} it can not have two consecutive odd operators and cannot have a factor $R(l)$ for $l\in L$ among even elements; if this happens, the length of the operator can be reduced.  It follows now from Lemma \ref{lem2.21}, (b), that $f(n)-1\geq f(n-1)$.  Since $f(0)\geq 1$, we have $f(n)\geq n+1$. On the other hand,  since the multiplicative length of $J$ is equal to $d$, we have $f(n)\leq d$.  This proves that if $n>{d-1}$ then the elements from $(J,L)_n$ can not be resqued, that is $(J,L)_{d}=(0)$.

\ctd

\begin{cor}\label{cor2.4}
$(R_{J}\la L\ra)^{2d}=(0)$ 
\end{cor}
\prf
 We have  $J(R_{J}\la L\ra)^{2d}\subseteq (J,L)_d=(0)$.  

\ctd
\begin{cor}\label{cor2.5}
$L^{4d+1}=(0)$ 
\end{cor}
\prf
We have $(L,L)_n=L^{\la 2n+1\ra}$,  hence by Lemma \ref{lem2.2'}
\bes
L^{4d+1}\subseteq L^{\la 2d+1\ra}=(L,L)_d=(0).
\ees
\ctd

\section{The case $L=(0)$}
\hspace{\parindent}
Consider first the case when $L=(0)$.

Let $f_1,\ldots,f_r$ be the identity elements of simple components of the algebra $J_{\0}$, $1=f_1+\cdots +f_r$.
In \cite[Proposition 6] {RZ}, it was proved that if $\dim J<\infty$  then $r\leq 2$. The condition $\dim J<\infty$  in fact was not used in the proof, therefore we have $r\leq 2$ in our case as well.  By corollary \ref{cor0} and \cite[lemma 7]{RZ},  only  the following two cases are possible:

(1) $r=1$,  $J_{\0}$ is an algebra of bilinear form or $\dim J_{\0}<\infty$,

(2) $r=2,\ J_{\1}=\{f_1,J_{\1},f_2\}$. In this case $J_{\0}$ is a sum of two simple algebras of bilinear form or a sum of a finite dimensional simple algebra and an algebra of bilinear form.

By \cite{ShZ}, it suffices to consider only the last case.  

\smallskip

Below for a Jordan superalgebra $J=J_{\0}\oplus J_{\1}$ we denote $J_{\0}=A,\ J_{\1}=M$, to avoid  use of  indices.
The product of two homogeneous elements $x,y\in J$ is denoted as $x\circ y$ unless $x,y\in J_{\1}$. In the latter case the product is denoted as $xy$.

\begin{lem}\label{lem5.1}
Let $J=A\oplus M$ be a finite dimensional Jordan superalgebra,  $A=Ff\oplus J(V),\ f^2=f, \ J(V)$ is a Jordan algebra of a symmetric form,  $M=\{f,M,e\}	$, where $e$ is the identity element of $J(V)$.  We do not assume that the form on $V$ is nondegenerate; we only assume that the annihilator of the form has codimention in $V$ $\geq 6$.  Then $M^2\circ V=(0)$.
\end{lem}
\prf
It follows from the classification of simple finite dimensional Jordan superalgebras \cite{Kac,Kantor}  that  the only simple superalgebra that has the even part $A$ of the form $Ff\oplus J(V)$ is the Kac superalgebra $K_{10}$.  But in $K_{10}$ we have $\dim V=5$, hence we conclude that $J$ can not be simple.

Let $I\idl J, I\neq J,$ be a maximal ideal of $J$.  Suppose that $I\cap A\neq (0)$. Then $f\in I$ or $e\in I$ or $I\cap A$ lies in the annihilator of the form on $V$. Passing to the simple superalgebra $J/I$ and using again the classification, we see that the third case is not possible.  Since $M=\{e,M,f\}$, it follows that $M\subseteq I$.   Suppose that $f\in I$.  Then $I=Ff+V'+M$, where $V'$ is the annihilator of the form on $V$.  In this case $M^2\subseteq Ff+V'$, hence $M^2\circ V=(0)$. 

Now let $e\in I,\, I=J(V)+M$. The classification of finite dimensional simple nonunital Jordan superalgebras in \cite{RZ} implies that the Jordan superalgebra $I$ can not be simple.  Let $I'$ be a maximal ideal of $I$. Then $I'\subseteq V'+M$. Since the superalgebra $I/I'$ is simple, the classification in \cite{RZ} implies $I'=M+V'$, hence $M^2\subseteq V'$  and $M^2\circ V=(0)$.  This completes the proof of the lemma.

\ctd

\begin{lem}\label{lem5.2}
As above, we assume that $A=Ff\oplus J(V),\ f^2=f, \ M=\{e,M,f\}$.  For abitrary elements $x,y\in M,\,v\in V$ we have
 $$
 (x\circ v)y+x (y\circ v)\in Ff+Fv.
 $$
\end{lem}
\prf
We have
\bes
(x\circ v)y+x(v\circ y)&=&v(R(x)R(y)-R(y)R(x)-R(xy))+(xy)\circ v\\
&=&\{x,v,y\}+(xy)\circ v.
\ees
The element $\{x,v,y\}$ lies in $Ff$. The vector space $A\circ v$ lies in $Fe+Fv$.  This completes the proof of the lemma.

\ctd
\begin{lem}\label{lem5.3}
Let $J=A+M$ be a Jordan superalgebra, $A=Ff+J(V),\,f^2=f,\, J(V)$ be a Jordan superalgebra of a symmetric nondegenerate form on $V$,  the space $V$ is infinite dimensional, $M=\{e,M,f\}$.  Then $M^2=(0)$.  
\end{lem}
\prf
Let $n\geq 3$. Choose elements $v_1,\ldots,v_{2n}\in V$ such that $v_i^2=e, \, 1\leq i\leq 2n,\,  v_{i}\circ v_{j}=0,\ 1\leq i\neq j\leq 2n$.  Consider also  the elements $u_1,\ldots,u_{2n}\in V, \ u_{2i-1}=\tfrac{1}{\sqrt2}(v_{2i-1}+\sqrt -1 v_{2i}),\, u_{2i}=\tfrac{1}{\sqrt2}(v_{2i-1}-\sqrt -1 v_{2i}),\ 1\leq i\leq n$. Then $u_i^2=0, \,1\leq i\leq 2n, \, u_{2i-1}\circ u_{2i}=e$, in all other cases $u_i\circ u_j=0$.

We have $u_{2i-1}D(u_{2i-1},u_{2i})=-u_{2i-1},\ u_{2i}D(u_{2i-1},u_{2i})=u_{2i}$. Hence the eigenvalues of $D(u_{2i-1},u_{2i})$ on $V$ are $0,\pm 1$. Since $D(v_{2i-1},v_{2i})=\sqrt{-1}D(u_{2i-1},u_{2i})$, the eigenvalues of $D(v_{2i-1},v_{2i})$ on $V$ are $0,\pm \sqrt{-1}$. 

Let $R_M(a),\, a\in A,$ denote  the right multiplication $R_M(a):M\rightarrow  M, \, x\mapsto xa$.  We have $R_M(v_i)^2=-\tfrac14,\ R_M(v_i)\circ R_M(v_j)=0, \, i\neq j$.
Therefore,  the subalgebra of $End_F(M)$ generated by $R_M(v_1),\ldots, R_M(v_{2n})$ is isomorphic to the Clifford algebra on the vector space $\sum_{i=1}^{2n}FR_M(v_i)$, which is isomorphic to the matrix algebra $M_{2^n}(F)$ (see \cite{Jac}).

For an arbitrary element $x\in M$  we have
\bes
xD(u_{2i-1},u_{2i})^2=-xD(v_{2i-1},v_{2i})^2=\tfrac14 x.
\ees
The element 
\bes 
E=\tfrac{1}{2^n}(Id+2D(u_{1},u_{2}))\cdots (Id+2D(u_{2n-1},u_{2n}))\in M_{2^n}(F)
\ees
is a primitive idempotent.   The right ideal $I$ of the algebra $M_{2^n}(F)$ generated by $E$ is an irreducible $M_{2^n}(F)$-module, and every irreducible $M_{2^n}(F)$-module is isomorphic to $I$.  

We have $ED(u_{2i-1},u_{2i})=\tfrac{1}{2}E$.
Let $S\subseteq M$  be the set of elements that belong to the eigenvalue $\tfrac{1}{2}$ with respect to  each derivation $D(u_{2i-1},u_{2i})$.
The $M_{2^n}$-module $M$ is a direct sum of isomorphiic $2^n$-dimensional irreducible modules.  Each of these modules is generated by an element from $S$. 

Choose arbitrary elements $x,y\in S$ and arbitrary indices $1\leq i_1,\ldots,i_k\leq 2n$. Consider the element $z=(xR(u_{i_1})\cdots R(u_{i_k}))y$. If the indices $i_1,\ldots,i_k$ are not all distinct then without loss of generality we can assume that $i_1=i_2$, then $xR(u_{i_1})^2=\tfrac12 xR(u_{i_1}^2)=0$. If $2i-1,2i\in \{i_1,\ldots,i_k\}$ for some $1\leq i\leq n,$ then without loss of generality we assume that  
$i_1=2i-1,\,i_2=2i$. We have
\bes
xR(u_{2i-1})R(u_{2i})&=&\tfrac12 x(D(u_{2i-1},u_{2i})+R(u_{2i-1})R(u_{2i})+R(u_{2i})R(u_{2i-1}))\\
&=&\tfrac12 x(D(u_{2i-1},u_{2i})+R(u_{2i-1}\circ u_{2i}))\\
&=&\tfrac12 x(D(u_{2i-1},u_{2i})+R(e))=\tfrac12 x.
\ees
Now suppose that the indices $i_1,\ldots,i_k$ are all distinct and for each $i,\, 1\leq i\leq n, $ no more than one element out of $u_{2i-1},u_{2i}$ lies among $u_{i_1},\ldots,u_{i_k}$. 

If the eigenvalue of $z$ with respect to at least one derivation $D(u_{2i-1},u_{2i})$ is nonzero, then $z\in \sum_{i=1}^{2n}Fu_i$.   If $zD(u_{2i-1},u_{2i})=0$ for all $i=1,\ldots, n$ then $z=(xR(u_{2})R(u_4)\cdots R(u_{2n}))y$. 
Hence $z\in \sum_{i=1}^{2n}Fu_i$ unless
$z=(xR(u_{2})R(u_4)\cdots R(u_{2n}))y$.

Note also that if $z$ has nonzero eigenvalues with respect to two derivations $D(u_{2i-1},u_{2i}),\, D(u_{2j-1},u_{2j}),\, i\neq j,$ then $z=0$. 

\smallskip
Consider the subspaces $Y=Fx+Fy, \, W=(YR(u_{2})R(u_4)\cdots R(u_{2n}))Y$.
Our aim is to prove that 
\bes 
J'=Ff+Fe+\sum_{i=1}^{2n}Fu_i+W+\sum_{p\geq 0}YR(u_{i_1})\cdots R(u_{i_p})=A'+M'
\ees
is a subsuperalgebra of $J$.

By lemma \ref{lem5.2},  
\bes
(M')^2\subseteq \sum (YR(u_{i_1})\cdots R(u_{i_p}))Y+Ff+Fe+\sum_{i=1}^{2n}Fu_i.
\ees
We proved above that 	
\bes
(YR(u_{i_1})\cdots R(u_{i_p}))Y\subseteq \sum_{i=1}^{2n}Fu_i+W.
\ees
Hence $(M')^2\subseteq A'$. It remains to show that $M'\circ W\subseteq M'$.

It is easy to see that
\bes
M'&=&Y+\sum YD(u_{i_1},u_{j_1})\ldots D(u_{i_k},u_{j_k})\\
&+&\sum YR(u_i)D(u_{i_1},u_{j_1})\ldots D(u_{i_k},u_{j_k}).
\ees
From $WD(u_{2i-1},u_{2i})=(0),\, 1\leq i\leq n,$ it follows that $W\subseteq Ff+Fe+\{u_1,\ldots,u_{2n}\}^{\perp}$. Hence $WD(u_{i},u_{j})=(0),\, 1\leq i,j\leq 2n.$  Therefore, it suffices to prove that 
\bes
(Y+\sum_{i=1}^{2n} YR(u_i))\circ W\subseteq M'.
\ees
We have
\bes
(YR(u_i))\circ W&=&((YR(u_{2})\cdots R(u_{2n}))Y)R(Y\circ u_i)\\
&\subseteq& ((YR(u_{2})\cdots R(u_{2n}))(Y\circ u_i))\circ Y\\
&+& YR(u_{2})\cdots R(u_{2n})D(Y,Y\circ u_i).
\ees
The subspace $(YR(u_{2})\cdots R(u_{2n}))(Y\circ u_i)$ belongs to the same eigenvalues with respect to $D(u_1,u_2),\ldots,D(u_{2n-1},u_{2n})$ as the element $u_i$. Hence 
\bes
(YR(u_{2})\cdots R(u_{2n}))(Y\circ u_i)\subseteq Fu_i,\ u_i\circ Y\subseteq M'.
\ees
The  subspace $u_jD(Y,Y\circ u_i)$ has nonzero eigenvalues with respect to $\geq 2$  derivations $D(u_{2p-1},u_{2p}),\,1\leq p\leq n$. Hence this subspace is zero.  Now, 
\bes
YD(Y,Y\circ u_i)\subseteq Y^2\circ (Y\circ u_i)+(Y(Y\circ u_i))\circ Y.
\ees
Both subspaces $Y^2$ and $Y(Y\circ u_i)$ also belong to nonzero eigenvalues with respect to $\geq 2$ derivations $D(u_{2p-1},u_{2p})$.  Hence 
\bes
Y^2=Y(Y\circ u_i)=(0).
\ees
We proved that $J'$ is a finite dimensional subsuperalgebra of $J$. By lemma \ref{lem5.1},  
\bes
(M')^2\circ v_i=(0),\, 1\leq i\leq 2n. 
\ees
Hence
\bes
((SR(u_{i_1})\cdots R(u_{i_p}))(SR(u_{j_1})\cdots R(u_{j_q})))\circ v_i=(0),
\ees
and therefore $M^2\circ v_i=(0)$.  This easily implies (taking $n$ sufficiently big) that  $M^2\subseteq Ff$.

Suppose that $x,z\in M,\, xz=\a f, \a\neq 0$. Since the vector space $V$ is infinite dimensional,  it follows that the vector space $M$ is infinite dimensional as well. Hence $M\ni y\neq 0$ such that $xy=zy=0$.  Now  by the Jordan identity we have
\bes
0=(xy,e,z)+(yz,e,x)+(zx,e,y)=\a(f,e,y)=-\tfrac{1}{4}\a y,
\ees
a contradiction.  This proves the lemma.

\ctd

\begin{prop}\label{L=0}
There does not exist a simple unital Jordan superalgebra $J=A+M$ with  $A=B\oplus J(V)$,  where $\dim V=\infty$ and $B$ is  a simple  finite dimensional Jordan algebra or the algebra $J(W)$ with $\dim W=\infty$.
\end{prop}
\prf
The case $B=J(W)$ was considered in \cite[theorem 3]{ShZ}, where it was proved that this case is impossible.  Note that the proof of this result in \cite{ShZ} contains a mistake; lemmas 4.4 and 4.5 of \cite{ShZ} should be substituted by lemma \ref{lem5.3} above.

\smallskip
Let now $B$ be a finite dimensional simple Jordan algebra with the identity element $f$.  It was proved in section 1 that we may assume that $f=f_1+\cdots +f_n$ is a sum of orthogonal idempotents $f_i$ such that $BU(f_i)=Ff_i$ for each $i=1,\ldots, n$.
Let $e$ be the identity element of $J(V)$. Then $M=\{e,M,f\}$.    Consider the subsuperalgebra $JU(e+f_i)=Ff_i+J(V)+M_i,\ M_i=\{e,M,f_i\}.$  By lemma \ref{lem5.3} we have $M_i^2=(0)$. Therefore, 
\bes
M^2=\sum_{i\neq j} M_iM_j\subseteq \sum_{i\neq j}\{f_i,M,f_j\}\subseteq B.
\ees
 Hence $B+M$ is an ideal of $J$.  This contradicts  simplicity of $J$.

\ctd

\begin{remark}\label{rem1}
Notice that in the proof of Proposition \ref{L=0} we did not use finiteness of multiplicative length of $J$.
\end{remark}

\section{The case $L\neq (0)$}
\hspace{\parindent}
Consider now the case when $L\neq (0)$.

Let again $1=e_1+\cdots +e_r$ be an orthogonal decomposition that corresponds to the decomposition of the quotient algebra $J_{\0}/L$ into a sum of simple algebras.
Consider the Pierce decomposition 
\bes
J=\oplus_{i\leq j=1}^r  \{e_i,J,e_j\},
\ees
and let $A_{ij}=  \{e_i,A,e_j\},  \ L_{ij}=  \{e_i,L,e_j\}, \ M_{ij}=  \{e_i,M,e_j\}$. Observe that if $i\neq j$ then $A_{ij}=L_{ij}\subseteq L$.

The following  two  lemmas were proved in \cite[lemmas 1.1.4, 1.1.5]{MarZel}.  We give the proof here for the sake of completeness.
\begin{lem}\label{lem7.1}
If $i\neq j$ then $ L_{ij}=(0)$.
\end{lem}
\prf
It suffices to prove that nonzero elements from $L_{ij}$ can not be rescued.  Assume that  it is not true and choose an element $0\neq a\in L_{ij}$ with  a rescuing operator $w=R(a_1)\cdots R(a_n)$ of minimal possible length.  Observe that no one of the elements $a_i$ lies in $L$.  In fact if some $a_i\in L$ then by lemma \ref{lem2.18} we have $w=\pm w_1R(a_i)R(a_n)$, where $a_n\in M$.
Note that  $aw_1\in M$,  hence by lemma \ref{lem2.17} $aw\in (LM)M\subseteq L$,  a contradiction.
 
 Clearly, $n\geq 3$.  If $a_1\in A$ then we may assume that $a_1\in A_{kl}$.  If $k\neq l$ then $a_1\in L$. We have shown that it  is impossible.  Hence  $a_1\in A_{ii}\cap A_{jj}$.  But then $aR(a_1)\in L_{ij}$, and this element  has a resquing operator of length $n-1$ which contradicts minimality of $n$.  Hence $a_1\in M$ and consequently $a_2\in A$.  Assume that $a_2\in A_{kk}$ with $k\not\in\{i,j\}$. Then $D(L_{ij},a_2)=0$ and $L_{ij}R(a_1)R(a_2)=L_{ij}R(a_1a_2)$, which again contradicts minimality of $n$.  
 
 Let $k=j$.  By \eqref{OSJ} we have the following comparisons  modulo $L$
 \bes
 L_{ij}w&\equiv&L_{ij}R(e_i)R(a_1)R(a_2)\cdots \\
 &\equiv&L_{ij}(-R(a_2)R(a_1)R_(e_i)+R(e_i)R(a_1a_2)+R(a_2)R(e_ia_1))\cdots \\
 &\equiv&L_{ij}R(a_2)R(a_1)R(e_i)\cdots\equiv L_{ij}R(a_1)R(e_i)\cdots.
 \ees
 Thus, without loss of generality, we may assume that $a_2=e_i$.  By  the skew-symmetry of $w$ on even factors we may assume that $a_{2k}\in\{e_i,e_j\}$ for all $k$.  Moreover,  repeating the above process for $e_i\in A_i$ instead of $a_2\in A_j$, we get
 \bes
 L_{ij} R(a_1)R(e_i)\subseteq =L_{ij}R(a_1)R(e_j)\cdots +L,
 \ees
 which implies that we may assume that $a_{2k}=e_j$ for all $k$, and hence
 \bes
 w=R(a_1)R(e_j)R(a_3).
 \ees
 Now, if $a_1\in M_{ks}$ with $\{k,s\}\cap\{i,j\}=\emptyset$ then $L_{ij}R(a_1)=(0)$.
 Furthermore, let $a_1\in M_{jk}+M_{ik},\,k\neq i,j$, then by \cite[(PD3), p.121]{Jac} we have 
 $\{a_1, L_{ij},e_j\}=(0)$,  which implies
 \bes
 L_{ij}R(a_1)R(e_j)\subseteq L_{ij}R(e_j)R(a_1)+L_{ij}R(a_1e_j),
 \ees
 a contradiction.  
 
 If $a_1\in M_{ii}+M_{jj}$ then $R(a_1)$ commutes with $R(e_j)$,  which gives a contradiction again.  
 
 The only remaining case is $a_1\in M_{ij}$.  In this case,  by the symmetry of odd $a_k$-s  we have also $a_3\in M_{ij}$,  and hence
   \bes
 L_{ij}R(a_1)R(e_j)R(a_3)&\subseteq& (L_{ij}M_{ij})R(e_j)R(a_3)
 \subseteq ((M_{ii}+M_{jj})R(e_j))M_{ij}\\
 &\subseteq& M_{jj}M_{ij}\subseteq A_{ij}\subseteq L,
 \ees
   which  again contradicts the resquing property of $w$.  
   
   This proves the lemma.
  
   \ctd
   
   \begin{lem}\label{lem7.2}
  If $L\neq (0)$ then $A/L$ is simple.
  \end{lem}
  \prf
  If $L\neq (0)$, then there exists $i$ such that $L_{ii}\neq (0)$.  We assume that $i=1$.
  
  We will show that $id_J\la L_{11}\ra\cap A\subseteq A_{11}+L$, and consequently $A/L\cong A_{11}+L/L\cong A_{11}/L_{11}$ is simple.
  
  Indeed, assume that assertion is not true.  Then there exists   an element $0\neq a\in L_{11}$  and an operator $w=R(a_1)\cdots R(a_n)$  such that $aw\in A$ but $aw\not\in A_{11}+L$.  Choose such an operator of minimal length.  
  
It is easy see that no two
consecutive elements  $a_k,a_{k+1}$ lie in $A$. 
Suppose that two consecutive elements $a_k, a_{k+1}$ lie in $M$.  If there 
is another odd element $u_i$ among $a_1,\ldots,a_n$ , then again using the Jordan operator identity, we can represent $w$ as a linear combination of operators
$\cdots R(a_k) R(a_{k+1} )R(a_i)\cdots$, $\cdots R(a_{k+1}) R(a_{k} )R(a_i)\cdots$ of length $n$ and operators of length  $<n$.  The identity \eqref{D0} implies that 
for odd elements $x, y, z \in M$, the operator $R(x)R(y)R(z)$ is a linear
combination of operators of length  2 and of operators of the type
$R(x)D(y, z)$.  By \eqref{D1} we can move $D(y, z)$ to the right end,  and it remains to
notice that $A_{11} D(M, M)\subseteq A_{11}$.

Suppose now that $a_k , a_{k+1}$ are the only odd elements among $a_1,\cdots,a_n$.  By \eqref{OSJ}, we can assume $k=1$. 
Then $(L_{11}M)M\subseteq L$ by Lemma \ref{lem2.17}.
  
The element $a_1$ clearly lies in $M$,  hence $a_2\in A$.  If $a_2\in A_{jj},\, j\neq 1$,  then the operators $R(a), R(a_2)$ commute, hence $aR(a_1)R(a_2)=aR(a_1a_2)$, a contradiction.  Thus $a_2\in A_{11}$.  Evidently, $a_1\in M_{1i}$.  If $i\neq 1$ then in view of the equality  $J_{1i}U(J_{11},J_{11})=(0)$ (see \cite[(PD3), p.121]{Jac}) we have 
  \bes
  aR(a_1)R(a_2)=-aR(a_2a_1)+R(a_2)R(a_1),
  \ees
  a contradiction.
  
  Hence we may assume that $a_1\in M_{11}$.  It is easy to see that for arbitrary $c\in  L_{11}$ the expression  $cw+A_{11}+L/A_{11}+L$ is symmetric on odd $a_k$-s and is skew-symmetric on even $a_{k}$-s.  Therefore, we can  assume that all $a_i\in A_{11}\cap M_{11}$.  Consequently,  $L_{11}w\subseteq A_{11}+L$, a contradiction.
  
  This proves the lemma.
  \ctd
 
 \medskip 
 
 Without loss of generality we may assume that $J$ is countable dimensional.
 Observe that any finite subset of elements of the factor-algebra $J _{\0}/L$ is
contained in a semisimple finite dimensional subalgebra. Therefore, the algebra $A=J _{\0}$ satisfies the conditions of \cite[Theorem 1]{Lyu},  and there exists a splitting $A = B\oplus L$, where $B\cong A/L$.

 We have to consider the two cases:
 
 (1) $A/L\cong H_n(C,*), \, n\geq 3$, where $(C,*)$ is a simple alternative algebra with involution $*$;
 
 (2) $A/L\cong J(V,f),$ the algebra of symmetric bilinear form  $f$ on the vector space $V$.
 
 \begin{lem}\label{lem7.3}
 Let $A/L\cong H_n(C,*), \, n\geq 3$, where $(C,*)$ is a simple alternative algebra with involution $*$.  Then either $J\cong H(\O_3)$ or $J\cong H_n(D,*)$, where $(D,*)$ is a simple associative superalgebra with a superinvolution $*$.
 \end{lem}
  \prf
 By the coordination theorem for superalgebras \cite{MSZ}, the superalgebra $J$ is isomorphic to the superalgebra of $*$-hermitian $n\times  n$ matrices $H(D_n,*)$, where $(D,*)$ is an alternative superalgebra with a superinvolution $*$ such that $*$-symmetric elements  lie in the associative center of $D$.  Since $J$ is simple, the superalgebra $D$ should be simple as well.  Recall that $char\,F=0$, hence by \cite{ShZ0} either $D=\O$, the algebra of octonions, or $D$ is associative.
 
 This proves the lemma.
 
 \ctd
 
 From now on we will assume that $A=J(V,f)+L$.
 
 \begin{lem}\label{lem7.4}
$tr((ML)A \cdot M)=(0).$
\end{lem}
  \prf
 Let $x,y\in M, \,a\in L,\, b\in A$  be Perice homogenous
elements such that $tr((xa)b\cdot y)\neq 0$.
Then, as in the proof of Lemma \ref{lem2.8}
\bes 
y=((xa)b)R(y)^2(\sum_{i\geq 0}R(u)^i), u\in L.
\ees 
If $i>0$ then $y_i=((xa)b)R(y)^2R(u)^i\in ML$,  hence by Lemma \ref{lem2.17} $((xa)b\cdot y_i\in M(ML)\subseteq L$  and $tr((xa)b\cdot y_i)=0$.  Therefore,  we may assume that 
$$
tr((xa)b\cdot y)=tr((xa)b\cdot ((xa)b)'),
$$
where $u'=uR(y)^2$.
 By Lemma \ref{lem2.16} $tr((xa)b\cdot y)=tr(xD(a,b)\cdot y)$,  and similary  
 $$
 tr((xa)b\cdot y)=tr(xD(a,b)\cdot (xD(a,b))').
 $$
We have 
 \bes
xD(a,b)\cdot (xD(a,b))'&=&(x\cdot (xD(a,b))')D(a,b)\\
&-&x\cdot (x'D(a,b)+xD(a',b)+xD(a,b'))D(a,b).
 \ees
 Clearly,  $tr((x\cdot (xD(a,b))')D(a,b))=0$.  Furthermore, for any $a,a_1\in L$ we have 
 \bes
 R(a)R(a_1)R(b)&=&-R(b)R(a_1)R(a)-R(ab\cdot a_1)\\
 &+&R(b)R(aa_1)+R(a)R(ba_1)+R(a_1)R(ab)\in R(A)R(L),
 \ees
 which easily implies that 
 \bee\label{idR(A)R(L)}
 D(a,b)D(a_1,b)\in R\la A\ra R(L).
 \eee
  Since $a'\in L$, we conclude that $D(a,b)^2, \, D(a,b)D(a',b)\in R(A)R(L)$,  and therefore by Lemma \ref{lem2.17}
  $$
 tr((xa)b\cdot y)=tr(xD(a,b)D(a,b')\cdot x).
 $$
 Repeating the process, and using the inclusion \eqref{idR(A)R(L)} and its linearization on $b$, we get
 \bes
 tr(xD(a,b)D(a,b')\cdot x)=tr(xD(a,b)D(a,b')\cdot xD(a,b)D(a,b')R(x)^2)\\
 =tr(xD(a,b)D(a,b')\cdot (xD(a,b)D(a,(b')'')+xD(a,b)D(a'',b')\\
 +xD(a'',b)D(a,b')+xD(a,b'')D(a,b')+x''D(a,b)D(a,b')),
 \ees
  where $u''=uR(x)^2$,  and hence 
 \bes
 tr((xa)b\cdot y)&=&\sum_{i=1}^4 tr(xD(a,b)D(a,b')D_{i1}D_{i2}\cdot x)\\ &+&tr(xD(a,b)D(a,b')D(a,b')D(a,b)\cdot x''),
 \ees
where for all $i$ one of $D_{i1},D_{i2}$ has form $D(a,b)$ or $D(a,b')$.  In view of \eqref{idR(A)R(L)} and Lemma \ref{lem2.17},
 the last trace expression is skew-symmetric on $b$-s,  hence is zero.  This completes the proof of the lemma.
 
 \ctd
  
 \begin{lem}\label{lem7.5}
 $(ML)A\cdot M\not\subseteq L$.
 \end{lem} 
 \prf
 If $(ML)A\cdot M\subseteq L$ then $I=L+(ML)\cdot A$ is an ideal of the superalgebra $J$.  To prove this we need only to check that $(ML)A\cdot A\subseteq (ML)A$.
 We have the following comparisons    modulo $(ML)A$
 \bes
 (ML)A\cdot A&\equiv& (ML,A,A)\subseteq (M,A,AL)+(L,A,AM)\\
 &\subseteq& (M,A,L)+(L,A,M)\equiv (0).
 \ees
 By the simplicity of $J$, we have $I=J$ and hence $L=A$, a contradiction.  
 This proves the lemma.
 
 \ctd
 \smallskip
 
 It follows from lemmas \ref{lem7.4} and \ref{lem7.5} that there exist elements $x,y\in M,\,a\in L,\, v,w\in V$ such that 
 $$
 tr(((xa)v) y\cdot w)\neq 0.
 $$
 \begin{lem}\label{lem7.31}
 The trace expression $tr(((xa)v) y\cdot w)$ is skew-symmetric on $v,w$.
 \end{lem}
 \prf
 By Lemma \ref{lem2.17} we have 
 \bes
 ((xa)v) y\cdot v\equiv (xa,v,y)v=\tfrac12 (xa,v^2,y)\equiv 0\ (mod\ L).
 \ees
 Therefore,  $tr(((xa)v) y\cdot v)=0$, which proves the lemma.
 
 \ctd
 
 \begin{prop}\label{prop7.1}
 Let $J=A+M$  be a simple Jordan superalgebra where $A=J(V)+L$ where $L^t=(0)$.  If $\dim\,V>2t+1$ then $L=(0)$ and $J$ is a superalgebra of a superform.
 \end{prop}
 
 Let  $n=t+1$ and $v_1,\ldots, v_{2n}\in V$ be an orthonormal system such that $v,w\in Span_F\la v_1,\ldots,v_{2n}\ra$.  	
 
 As in \cite{ShZ} we consider  the action of the group $(\Z/2\Z)^{2n}$ generated by the operators $U(v_1),\ldots,U(v_n)$ on $A,M,L$. We consider also the action of  $(\Z/2\Z)^{2n}$ on the Clifford algebra $Cl(v_1,\ldots,v_{2n})$. 
 
 For an element $x\in J$ we write $	x\sim v_{i_1}\cdots v_{i_k}\in Cl(v_1,\ldots,v_{2n}),\,i_1<\cdots< i_k,$  if $x$ and $v_{i_1}\cdots v_{i_k}$ belong to the same eigenfunctional with respect to the action of the group  $(\Z/2\Z)^{2n}$.  For an eigenvector $x$ we denote $|x|=v_{i_1}\cdots v_{i_k}.$
 Without loss of generality we assume that $v=v_1, w=v_2; x,y,a$ are eigenvectors and
  \bes
  tr(((xa)v_1) y\cdot v_2)=1.
  \ees
 We say that an eigenvector $x\in J $ anticommutes with $v_i$ if $x\cdot v_i=0$, which is equivalent to $U(v_i)(x)=-x$.  The vector $x$ commutes with $v_i$ if $D(x,v_i)=0$,  which is equivalent to $U(v_i)(x)=x$. 
 
Lemmas  \ref{lem7.4} and \ref{lem7.5}  imply that $x,y,a$ anticommute with $v_1$ and $v_2$.  
 
 \begin{lem}\label{lem7.6}
 $a\sim v_1v_2; \, x,y\sim v_1,\cdots v_{2n}$.
 \end{lem}
 \prf
 We need to show that the element $a$ commutes with all $v_i,\,i\geq 3$. 
 Suppose that $a\cdot v_i=0$.  We have
 \bes
 (((xa)v_1) y)\cdot v_2=aD(x,v_1)D(y,v_2) \ (mod\ L).
 \ees
 The product of eigenvalues of $a,x,y$ with  respect to $U(v_i)$ is 1.  Hence one of the elements $x,y$ anticommutes with $v_i$, the other one commutes with $v_i$.
 Suppose that $x$ anticommutes with $v_i$,  $y$ commutes with $v_i$.  
 
 We will show that 
 \bes
 tr(aD(v_1,v_i)D(x,v_i)D(y,v_2))\neq 0.
 \ees
 That contradicts $a\cdot v_i=0$.
 
 We have $|x||y||a|v_1v_2=\pm 1$, hence $|x||y||a|=\pm v_1v_2$.  Now,
 \bes
 |aD(x,v_i)D(y,v_2)|=\pm |a||x||y|v_iv_2=v_1v_i,
 \ees
 hence $aD(x,v_i)D(y,v_2)\in L$.
 
 Therefore,  we have
 \bes
 aD(v_1,v_i)D(x,v_i)D(y,v_2)=a[D(v_1,v_i),D(x,v_i)D(y,v_2)] \ (mod\ L)
 \ees
 Furthermore,
 \bes
 [D(v_1,v_i),D(x,v_i)]&=&-D(x,v_1),\\ \
 [D(v_1,v_i),D(y,v_2)]&=&-D((yv_i)v_1,v_2)=D((yv_i)v_2,v_1).
 \ees
 Hence
\bes 
 aD(v_1,v_i)D(x,v_i)D(y,v_2)=-a(D(x,v_1)D(y,v_2)-D(x,v_i)D((yv_i)v_2,v_1)).
 \ees
 By Lemma \ref{lem7.31} we have
 \bes
 &tr(aD(x,v_i)D((yv_i)v_2,v_1))=-tr(aD(x,v_1)D((yv_i)v_2,v_i))&\\
 &=-tr(aD(x,v_1)D((yv_i)v_i,v_2))=-tr(aD(x,v_1)D(y,v_2)).&
 \ees
 Therefore,
 \bes
 tr(aD(v_1,v_i)D(x,v_i)D(y,v_2))&=&-2\, tr(a(D(x,v_1)D(y,v_2))\\
 &=&-2\, tr(((xa)v_1) y\cdot v_2)=-2,
 \ees
 a contradiction. We showed that $a\sim v_1v_2$.
 
 Now we have to show that $x$ and $y$ anticommute with any $v_i,\, i\geq 3$.
 Suppose that $x$ commutes with $v_i,\,x'=x\cdot v_i,\,x=x'\cdot v_i$. Then $x\cdot a=x'R(v_i)R(a)=(x'a)v_i$.
 
 By the Jordan identity and Lemma \ref{lem2.17} we have the following comparisons modulo $L$
 \bes
 (xa)v_1\cdot y&=& ((x'a)v_i)v_1\cdot y=(x'a,v_i,v_1)y\\
& \equiv& (x'a,v_iy,v_1)-v_i(x'a,y,v_1)\equiv 0,
\ees
 a contradiction.   We showed that $x\sim v_1\cdots v_{2n}$.  Since $(xa)v\cdot y=(ya)v\cdot x \ (mod\ L),$ we conclude that $y\sim v_1\cdots v_{2n}$ as well.
 This completes the proof of the lemma.
 \ctd
 
 Since the expression $(xa)v \cdot y$ is symmetric in $x,y$ modulo $L$, without loss of generality we will assume that $x = y$.

 Let $a_{12}=a$. For $3\leq i\neq j\leq 2n$ let
 \bes
 a_{ij}=a_{12}D(v_1,v_i)D(v_2,v_j),\ a_{ij}\sim v_iv_j.
 \ees
 It is easy to see using the fact that $x$ anticommutes with all $v_k$  that 
 \bes 
 a_{ij}D(x,v_i)D(x,v_j)=1\ (mod\ L).
 \ees
 We also notice that 
 \bes
 D(x,v_i)D(x,v_j)+D(x,v_j)D(x,v_i)=0
 \ees
 for all $i,j$.
 
 \begin{lem}\label{lem7.7}
 (a) Let $1\leq i_1<\cdots<i_{2m}\leq 2n, \ m<n-1$. Then
 $
 a_{ij}D(x,v_{i_1})\cdots D(x, v_{i_{2m}})\in L
 $
 unless $\{i_1,\ldots,i_{2m}\}=\{i,j\}$;\\
 (b) If $1\leq i_1\leq\cdots\leq i_{2m+1}\leq 2n$  then $a_{ij}D(x,v_{i_1})\cdots D(x,v_{i_{2m+1}})\in M\cdot L$ unless $\{i_1,\ldots,i_{2m+1}\}=\{i\}$ or $\{j\}$.
 \end{lem}
 \prf
 (a) $a_{ij}D(x,v_{i_1})\cdots D(x,v_{i_{2m}})\sim v_iv_jv_{i_1}\cdots v_{i_{2m}}$.
 Since the length of this product is even, it can not be equal to $\pm v_k,\ 1\leq k\leq 2n.$ Since the length is less then $2n$, it can not be equal to $v_1\cdots v_{2n}$.
 Hence unless $v_iv_jv_{i_1}\cdots v_{i_{2m}}=\pm 1$, the element $a_{ij}D(x,v_{i_1})\cdots D(x,v_{i_{2m}})$ lies in $L$. This completes the proof of the assertion (a).
 
 If $\{i_1,\ldots,i_{2m+1}\}\neq\{i\}$ or $\{j\}$, then there exists $k\in\{i_1,\ldots,i_{2m+1}\},\ k\not\in\{i,j\}$. We have
 \bes
 a_{ij}D(x,v_{i_1})\cdots D(x,v_{i_{2m+1}})=\pm a'D(x,v_k),
 \ees
 where $a'= a_{ij}D(x,v_{i_1})\cdots \widehat{D(x,v_k)}\cdots D(x,v_{i_{2m+1}})$.
 We have 
 $$
 a'\sim v_iv_jv_{i_1}\cdots\widehat{v_{k}}\cdots v_{i_{2m+1}}, 
 $$
 hence $a'$ commutes with $v_k$.  If $\{i_1,\ldots,i_{2m+1}\}\setminus\{v_k\}\neq\{i,j\}$ then $a'\in L$ by (a).  We have 
 \bes
 a'D(x,v_k)=-(a'\cdot v_k)\cdot x\in M\cdot L.
 \ees
 Let $\{i_1,\ldots,i_{2m+1}\}=\{i,j,k\}$.  Then $a'=1+b,\ b\in L,\, b\sim 1$.  Again 
 $a'D(x,v_k)=bD(x,v_k)=-(bv_k)x\in M\cdot L$.  
 
 This completes the proof of the lemma.
 
 \ctd
 
 \begin{lem}\label{lem7.8}
 \bes
 a_{12}R(a_{34})\cdots R(a_{2t-1,2t})D(x,v_1)\cdots D(x,v_t)=1\ (mod\ L).
 \ees
 \end{lem}
 \prf
 For a subset $\pi=\{1\leq i_1<\cdots<i_m\leq 2t\}$ denote $D_{\pi}=D(x,v_{i_1})\cdots D(x,v_{i_m})$. Then
 \bes
 a_{12}R(a_{34})\cdots R(a_{2t-1,2t})D(x,v_1)\cdots D(x,v_t)\\
=\sum (a_{12}D_{\pi_1})R(a_{34}D_{\pi_2})\cdots  R(a_{2t-1,2t}D_{\pi_t}),
\ees
where the summation is done over decompositions $\pi_1\dot\cup\pi_2\dot\cup\cdots\dot\cup\pi_d=\{1,2,\ldots,t\}$. By Lemma \ref{lem7.7} $a_{2t-1,2t}D_{\pi_t}\in L$ or $a_{2t-1,2t}D_{\pi_t}\in M\cdot L$ or $\pi_t=\{2t-1,2t\}$.
In the first two cases 
$$
 (a_{12}D_{\pi_1})R(a_{34}D_{\pi_2})\cdots  R(a_{2t-1,2t}D_{\pi_t})\in L.
 $$
   In the third case this element is equal to
$$
 (a_{12}D_{\pi_1})R(a_{34}D_{\pi_2})\cdots  R(a_{2t-3,2t-2}D_{\pi_{t-1}})\ (mod\  L).
 $$ 

We repeat the above argument for the element $a_{2t-3,2t-2}D_{\pi_{t-1}}$ and so on.  After $t-1$ steps we get 
\bes
 a_{12}R(a_{34})\cdots R(a_{2t-1,2t})D(x,v_1)\cdots D(x,v_t)\\
 \equiv a_{12}D(x,v_1)D(x,v_2)\equiv 1\ (mod\ L).
 \ees
  This completes the proof of the lemma.

\ctd

It follows from Lemma \ref{lem7.8} that $a_{12}R(a_{34})\cdots R(a_{2t-1,2t})\neq (0)$, hence $L^t\neq (0)$, a contradiction.

This proves the Proposition \ref{prop7.1}.

\ctd

\section{The case $L\neq (0),\ \dim V<\infty$.}
\hspace{\parindent}
We continue our consideration of the simple unital Jordan superalgebra $J=A+M$ where $A=L+J(V),\  L\triangleleft A,\ L^{4d+1}=(0)$.  In this section we will prove 

\begin{prop}\label{prop8.1}
If $\dim_F V<\infty$ then $\dim_F A<\infty$.
\end{prop}

Notice that the case $\dim A<\infty$ was considered in \cite{ShZ}. In this case $\dim J<\infty$ as well.

\begin{lem}\label{lem8.1}
The odd part $M$ is a one-generated bimodule over $A$.
\end{lem}
\prf
We have proved in the previous section that there exist $a\in L,\, x\in M,\, v_1,v_2\in V$ such that 
\bes
aD(x,v_1)D(x,v_2)=1 +u,\ u\in L.
\ees

Let $M'$ be the $A$-sumodule of $M$  generated by $x$.  Then $AD(x,A),  AD(x,A)\subseteq M'$. Hence for any $m\in M=M(1+u)$ we have
\bes
M&=&M(1+u)\subseteq AD(x,A)D(x,A)R_M\subseteq AD(x,A)R_MD(x,A)\\&+&AD(x,A)R_{(x,M,A)}
\subseteq 
AD(x,A)+AD(x,A)R_A\subseteq M'.
\ees
\ctd

\begin{lem}\label{lem8.2}
$M=F\cdot x +(ML)A$.
\end{lem}
\prf
Notice first that $(ML)A$ is an $A$-submodule of $M.$  In fact, we have the following comparisons  modulo the subspace $(ML)A$:
\bes
(ML)A\cdot A\equiv  (ML,A,A)\subseteq (M,A,LA)+(L,A,MA)\equiv (0).
\ees
The factor-module $\overline M=M/(ML)A$ is in fact a module over the  algebra $J(V)=A/L$.   Let us show that $\overline M\cdot V=(0)$.  It follows from Lemmas \ref{lem7.4} and \ref{lem7.5} that $(ML)A\cdot M\not\subseteq L$ and $tr((ML)A\cdot M)=0$.
Therefore,  $(ML)A\cdot M\subseteq V+L$.  Since $(ML)A\cdot M$ is invariant with respect to $D(V,V)$ and $V$ is irreducible with respect to $D(V,V)$ we have that 
$V\subseteq (ML)A\cdot M$.  Now by Lemma \ref{lem2.17} we have modulo the subspace $(ML)A$
\bes
 ((ML)A\cdot M)M&\subseteq& ((ML)A\cdot A+((ML)A,M,M)\\
&\equiv& (ML,M,MA)+(A,M,M(ML))\\
&\equiv& (A,M,L)\equiv (0).
\ees
This proves that $M\cdot V\subseteq (ML)A$.  In particular,  $x\cdot J(V)\subseteq F\cdot x+(ML)A$ and 
$$
x\cdot A=x\cdot (J(V)+L)\subseteq  F\cdot x+(ML)A.
$$
Since $M$ is generated by $x$ as an $A$-module, this proves the lemma.

\ctd

Let $R_M\la A\ra$  be the subalgebra of $End\, M$
generated by multiplications $R(a): M\rightarrow M, \, a\in A.$
Denote by $I$ the ideal of $R_M\la A\ra$ generated by $R(L)$.

\begin{lem}\label{lem8.3}
The ideal $I$ is nilpotent.
\end{lem}
\prf
Let $I_k,\,k\geq 1$ be the ideal of $R_M\la A\ra$ generated by $R(L)^k$.  By Corollary 
\ref{cor2.4} $I_{2d}=(0)$.  Let us prove that for any $s\geq 1$
\bes
I^s\subseteq (R(A)R(L))^sR_M\la A\ra.
\ees
By the Jordan operator identity \eqref{OSJ} we have $I=R(A)R(L)R_M\la A\ra.$ Assume that 
\bes
I^{s-1}\subseteq (R(A)R(L))^{s-1}R_M\la A\ra.
\ees
Then
\bes
I^s&\subseteq&(R(A)R(L))^{s-1}R_M\la A\ra I\subseteq (R(A)R(L))^{s-1} I\\
&=&(R(A)R(L))^{s-1} R(A)R(L)R_M\la A\ra=(R(A)R(L))^sR_M\la A\ra.
\ees
For arbitrary elements $u_1,\ldots,u_s\in L;\,a_1,\ldots,a_{s-1}\in A$ the operator
\bes
R(u_1)R(a_1)R(u_2)\cdots R(a_{s-1})R(u_s)
\ees
is skew-symmetric on $a_1,\ldots,a_{s}$ modulo $I_2$. Choosing $s=\dim V+3$, we get
\bes
I^{\dim V+3}\subseteq I_2.
\ees
Let us prove now that
\bee\label{id8.1}
I_2\subseteq R(L)R_M\la A\ra.
\eee
Notice first that for any $a,b\in A,\, n\in L$ we have by \eqref{D0}
$$
[D(a,b),R(n)]=R((a,n,b))\in R(L),
$$
hence $D(a,b)R(L)^2\subseteq R(L)^2R_M\la A\ra$.  Now, by  \eqref{D2}
any operator $R(a_1)R(a_2)\cdots R(a_n), \,n\geq 3,$ can be represented as a sum of
operators of the form $R(b_1)R(b_2)D_1\cdots D_k$,  where every $D_i=D(c_i,d_i)$ for some $b_i,c_i,d_i\in A$.

Therefore, it suffices to prove that 
\bes 
R(A)^2R(L)^2\subseteq R(L)R_M\la A\ra.
\ees
The Jordan operator identity implies
$$
R(A)R(L)^2\subseteq R(L)^2R(A)+R(L^2)R(A). 
$$
 Repeating the process, we get 
\bes
R(A)^2R(L)^2\subseteq R(L)R\la A\ra+R(A)R(L^2)R(A).
\ees
Finally,  by linearization of the operator identity $R(a)R(a^2)=R(a^2)R(a)$ we have
\bes
R(A)R(L^2)\subseteq R(L)^2+R(L^2)R(A),
\ees
which finishes the proof of \eqref{id8.1}.

The identity \eqref{id8.1} implies that for any $k\geq 1$
\bes 
I_kI_2&\subseteq &R_M\la A\ra R(L)^kR_M\la A\ra R(L)^2R_M\la A\ra\\
&\subseteq &R_M\la A\ra R(L)^k R(L)R_M\la A\ra\subseteq I_{k+1}.
\ees
Therefore,  $I_2^{2d-1}\subseteq I_{2d}=(0),$ and
\bes
I^{(\dim V+3)(2d-1)}\subseteq I_2^{2d-1}=(0).
\ees

\ctd

\begin{lem}\label{lem8.4}
$((ML)A)^2\subseteq L$.
\end{lem}
\prf
By Lemma \ref{lem2.17} 
\bes
((ML)A)^2=(MD(L,A))^2=MD(L,A)^2\cdot M \ (modulo\ L).
\ees
Assume that $MD(L,A)^2\cdot M\not\in L$.  It follows from Lemma \ref{lem7.4} and the invariance of $MD(L,A)^2\cdot M$  under all inner derivation that in this case $MD(L,A)^2\cdot M=V$ modulo $L$.

For abitrary elements $x\in M, \,u\in L,\, a\in A,$ we let $a=v+\alpha 1,\,\alpha\in F$. Then $v\in MD(L,A)^2\cdot M+L$ and we have
\bes
(xu)a=(xu)v\ (mod\ ML).
\ees

Hence	,
\bes
(ML)A\subseteq MD(L,A)^2R(M)R(ML)+ML.
\ees
We have
\bes
R(M)R(ML)\subseteq D(M,M)+R(ML)R(M),
\ees
hence by Lemma \ref{lem2.17}
\bes
(ML)A&\subseteq& MD(L,A)^2+(MD(L,A)^2)(ML)\cdot M+ML\\
&=&MD(L,A)^2+ML.
\ees
We proved that 
\bes
MD(L,A)\subseteq MD(L,A)^2+ML.
\ees
Hence
\bes
MD(L,A)\subseteq MD(L,A)^k+ML \hbox{ for any } k\geq 2.
\ees
By Lemma \ref{lem8.3} $MD(L,A)\subseteq ML,\ (ML)A=ML$.
This implies that $ML+L$ is an ideal in $J$, a contradiction.

This completes the proof of the Lemma.

\ctd

{\bf Proof of the Proposition \ref{prop8.1}}\\

Let $e_1,\ldots,e_n$ be a basis of $J(V)$, and let $x\in M$ be an element such that $M=F\cdot x +(ML)A$. 
Consider the finite set $\mathcal P$ of  operators
\bes
R(x)R(a_1)R(x)\cdots R(a_k)R(x),
\ees
where $a_j$-s are distinct elements from $\{e_1,\ldots,e_n\}$. 
It follows from  Lemmas \ref{lem2.18} and \ref{lem8.4} that for any $u\in L$ the resquing operator for  $u$ is a scalar multiple of some operator in $\mathcal P$.  In other words, for any $u\in L$ there exists $p\in\mathcal P$ such that $up\in A\setminus L$.

For any $p\in \mathcal P$  consider the subspace 
$L_p=\{u\in L\,|\,up\in L\}.$ Since $\dim A/L<\infty, $ it follows that $L_p$ has a finite codimension in $L$.  Hence $\cap_{p\in\mathcal P}L_p$ has finite codimension in $L$.  It was noticed above that $\cap_{p\in\mathcal P}L_p=(0)$. Hence $\dim L<\infty, \ \dim A<\infty,$ and  by \cite{ShZ} $\dim J<\infty$.

This finishes the proof of the Proposition \ref{prop8.1}.
\ctd
\medskip

Now the Main Proposition  follows from Corollories \ref{cor0}, \ref{cor2.3}, \ref{cor2.5},  Lemma \ref{lem7.3}, and Propositions \ref{L=0}, \ref{prop7.1},  \ref{prop8.1}, and the results of  \cite{ShZ}.\\[3mm]

{\bf Proof of the Main Theorem.}\\

Let $J$ be a simple unital Jordan superalgebra over an algebraically closed field $F$. 
We may assume that $X(J)\neq (0)$, otherwise $J\in Var$.   Then $J=X(J)$ and by \cite{Zel5} $J$ has a finite multiplicative length.  It is clear that $J$ is primitive, hence by the Main Proposition either $S(J)=(0)$ or $J$ is isomorphic to $H(\O_3)$  or to $K_{10}$.
Notice that the $T$-ideal $X$ lies in the {\em hermitian} or {\em teatrad-eater} ideal $T$ introduced in \cite{Zel2}.  The ideal $T$ has the following property: if $B$ is an $i$-special Jordan (super)algebra in which $T(B)\neq (0)$ then there exists an associative (super)algebra $C$ with (super)involution $*$ such that  $T(B)\cong H(C,*)$.  Since $J=G(J)=T(J)$,   if $S(J)=(0)$ then we have $J\cong H(C,*)$ for some associative superalgebra with superinvolution $(C,*)$.

\ctd


\begin{thebibliography}{99}

\bibitem{CanKac} Cantarini, \,N.  and Kac,\,V. G.:  Classification of linearly compact simple Jordan and generalized Poisson superalgebras,  J. Algebra,  313, 100--124 (2007).
\bibitem{ChengKac} Cheng, S.J.  and Kac, V.G.: A new $N=6$ superconformal algebra, Comm. Math. Phys 186 (1997),  no. 1,  219--231.
\bibitem{Glennie} Glennie, C. M.: Some identities valid in special Jordan algebras but not valid in all Jordan algebras, Pacific J. Math. 16 (1966), 47--59.
\bibitem{GA} G\'omez-Ambrosi, C. : 
On the simplicity of Hermitian superalgebras,  Nova J. Algebra Geom. 3, no 3, 193--198 (1995).
\bibitem{GA-M} G\'omez-Ambrosi, C.,  Montaner, F. : On Herstein's constructions relating Jordan and associative superalgebras, Comm. Algebra 28 (2000),  no. 8, 3743--3762.
\bibitem{GLS} Grozman, P.,  Leites, D.  and Shchepochkina, I.: Lie superalgebras of string theories,  Acta Math. Vietnam 26 (2001),  no. 1,  27--63.
\bibitem{Jac} Jacobson, Nathan: {\em Structure and representations of Jordan algebras.}  American Mathematical Society,  Providence, RI, 1968, x+453 pp.
\bibitem{Jac2} Jacobson, Nathan: {\em Structure  of Rings.}  American Mathematical Society,  Providence,  RI, 1956, x+299 pp.
\bibitem{Kac}  Kac, V. G.: Classification of simple $\Z$-graded Lie superalgebras and simple Jordan superalgebras,  Comm. Algebra, 5,  no. 13, 1375--1400 (1977).
\bibitem{KMZ} Kac, V. G., Martinez, C., and Zelmanov E.: Graded Simple Jordan Superalgebras of Growth One,  Amer. Math. Soc., Providence (2001) (Mem. Amer. Math. Soc.; V. 150,  no. 711).
\bibitem{Kantor} Kantor, I.L.: Jordan and Lie superalgebras determined by a Poisson algebra, in: Algebra and Analysis (Tomsk, 1989),  pp. 55--80; Amer. Math. Soc. Transl. Ser. 2, 151,  Amer. Math. Soc., Providence, 1992.
\bibitem{Kap1} Kaplansky, I.: Superalgebras,  Pacific J. Math. 86 (1980), no.1, 93--98.
\bibitem{Kap2} Kaplansky,  I.: Graded Jordan algebras I, preprint.
\bibitem{KingMC} King, D.,  McCrimmon, K.: The Kantor construction of Jordan superalgebras,  Comm. Algebra,  20,  no. 1, 109--126 (1992).

\bibitem{Lyu} Lyu, Shao-syu\`e.: On the splitting of infinite algebras.  Mat. Sb. (N.S.) 42/84 (1957), 327--352.
\bibitem{MSZ} Martinez, C.,  Shestakov, I., Zelmanov, E.:  Jordan bimodules over the superalgebras  $P(n)$  and  $Q(n)$.
Trans.  Amer.  Math.  Soc. 362 (2010),  no. 4,  2037--2051.
\bibitem{MarShZ} Martinez, C.,  Shestakov, I.,  Zelmanov, E.: Jordan superalgebras defined by brackets. J. London Math. Soc. (2) 64(2001),  no.2, 357--368.

\bibitem{MarZel} Martinez, C.,  Zelmanov, E.: Simple finite-dimensional Jordan superalgebras of prime characteristic.  J. Algebra 236(2), 575--629 (2001).
\bibitem{McC} McCrimmon, K. : Speciality and non-speciality of two Jordan superalgebras.  J. Algebra 149,  no. 2, 3 26--351 (1992).

\bibitem{MedZel}  Medvedev, Yu. A., Zelmanov,  E. I.: Solvable Jordan algebras.  Comm. Algebra 13 (1985), no. 6, 1389--1414.
\bibitem{MedZel1}  Medvedev, Yu. A., Zelmanov,  E. I. : Some counterexamples in the theory of Jordan algebras.  Nonassociative algebraic models (Zaragoza, 1989), 1--16. Nova Science Publishers, Inc.,  Commack, NY, 1992

\bibitem{NS} Neveu, A.,  Schwarz,J.H. :
Factorizable dual models of pions
Nucl. Phys. B, 31 (1971),  86--112.

\bibitem{RZ} Racine, M. L., Zelmanov, E. I.:  Simple Jordan superalgebras with semisimple even part.  J. Algebra 270 (2003), no. 2, 374--444.

\bibitem{Ram} Ramond, P.:
Dual theory for free fermions
Phys. Rev. D, 3 (1971),  2415--2418.

\bibitem{Sh5} Shestakov, I. P. : Superalgebras and Counterexamples, Sibirsk. Matem. Zhurnal 32 (1991), no.6, 187--196; English transl. in Siberian Math. J. 32 (1991), no.6, 1052--1060 (1992).
\bibitem{Sh3} Shestakov, I. P. : Prime alternative superalgebras of arbitrary characteristic.
(Russian) Algebra i Logika,  36, no. 6 (1997), 675--716; English transl.: Algebra and Logic 36, no. 6 (1997), 389--420.
\bibitem{Sh4} Shestakov, I. P. : Quantization of Poisson superalgebras and speciality of Jordan Poisson superalgebras,  Algebra and Logic (5) (1993) 309--317.

\bibitem{ShZ} Shestakov, I. P., Zelmanov, E. I.: Simple Jordan superalgebras with the even parts of Clifford type,  Contemporary Mathematics,  to appear; arXiv:2503.08164.
\bibitem{Zel2} Zelmanov, E. I.: On prime Jordan algebras.(Russian) Algebra i Logika.18 (1979),  no.2, 162--175.
\bibitem{Zel6} Zelmanov, E. I.: On prime Jordan algebras II.(Russian) Siberian Math. J.  24 (1983) 89--104.
\bibitem{Zel1} Zelmanov, E. I.: Absolute zero divisors and algebraic Jordan algebras.(Russian) Siberian Math. J.,23 (1982),  no.6, 100--116, 206.
\bibitem{Zel4} Zelmanov, E. I.: Characterization of the McCrimmon radical. (Russian) Siberian Math. J.,  25  (1984), No. 5 (147), 190--192.
\bibitem{Zel5} Zelmanov, E. I.: , Birepresentations of infinite-dimensional Jordan algebra. (Russian) Siberian Math. J., 27:6 (1986), 79–94; English transl. in Siberian Math. J., 27:6 (1986), 849–862.  
\bibitem{ShZ0} Zelmanov, E. I. , Shestakov, I.P.: Prime alternative superalgebras and the nilpotency of the radical of a free alternative algebra.(Russian)
Izv.  Akad.  Nauk SSSR Ser. Mat. 54 (1990), no. 4, 676--693; translation in
Math. USSR-Izv. 37 (1991), no. 1, 19--36.
\bibitem{Zhel2} Zhelyabin, V. N.: New examples of simple Jordan superalgebras over an arbitrary field of characteristic zero.  St. Petersburg Math. J., 24,  no. 4, 591--600 (2013).
\bibitem{Zhel4} Zhelyabin, V.N.: Examples of prime Jordan superalgebras of vector type and superalgebras of Cheng-Kac type,  Siberian Math. J., 54, no.1  (2013),  33--39.
\bibitem{ZhSh} Zhelyabin, V. N.  and Shestakov, I. P.: Simple special Jordan superalgebras with associative even part,  Sib. Math. J., 45,  no. 5, 860--882 (2004).

\bibitem{ZhZ} Zhelyabin, V. N.  and Zakharov, A. S. : The superalgebras of Jordan brackets defined by the $n$-dimensional sphere,  Siberian Math. J., 61, no. 4 (2020), 632--647.

\bibitem{ZSSS}  Zhevlakov, K. A.; Slin'ko, A. M.; Shestakov, I. P.; Shirshov, A. I.: {\em Rings that are nearly associative.} (Russian) Moscow, Nauka, 1978; English transl. 
Academic Press, Inc. , New York-London, 1982, xi+371 pp.

\end{thebibliography}
 \end{document}